\documentclass[twoside. letter]{amsart}

\usepackage{amsmath,amssymb,latexsym}
\usepackage{exscale}
\usepackage{cite}
\usepackage{epsfig}
\usepackage{amscd}
\usepackage{graphics}
\usepackage{color}
\usepackage{dsfont}

\renewcommand{\geq}{\geqslant}
\renewcommand{\leq}{\leqslant}

\newcommand{\el}{\lambda^\epsilon}

\newcommand{\eP}{\Phi^\epsilon}

\newcommand{\A}{(\lambda)}

\newcommand{\Pf}{(-p)}
\newcommand{\OO}{\mathcal{O}_\delta}
\newcommand{\Ob}{\overline{\mathcal{O}}_\delta}
\newcommand{\C}{\mathcal{C}^\epsilon_\delta}
\newcommand{\Cd}{\mathcal{C}'_\delta}
\newcommand{\Sde}{\mathcal{S}^\epsilon_\delta}
\newcommand{\Ss}{\mathcal{S}_\delta}

\newcommand{\Cdj}{\mathcal{C}'_{\delta_j}}

\newtheorem{theorem}{Theorem}[section]
\newtheorem{lemma}[theorem]{Lemma}
\newtheorem{proposition}[theorem]{Proposition}
\newtheorem{remark}[theorem]{Remark}

\newtheorem{corollary}[theorem]{Corollary}

\newtheorem*{main-theorem}{Main Theorem}

\newtheorem*{remark*}{Remark}

\numberwithin{equation}{section}

\title{Stokes Waves with Vorticity}

\author{Vera~Mikyoung~Hur}
	\address{Department of Mathematics, University of Illinois at Urbana-Champaign,
	Urbana, IL 61801}
	\email{verahur@\allowbreak math.\allowbreak uiuc.\allowbreak edu}

\begin{document}

\maketitle

\begin{abstract}
The existence of periodic waves propagating downstream on the surface of a two-dimensional 
infinitely deep water under gravity is established for a general class of vorticities.
When reformulated as an elliptic boundary value problem in a fixed semi-infinite strip with a parameter,
the operator describing the problem is nonlinear and non-Fredholm.
A global connected set of nontrivial solutions is obtained via singular theory of bifurcation.
Each solution on the continuum has a symmetric and monotone wave profile. 
The proof uses a generalized degree theory, global bifurcation theory
and Wyburn's lemma in topology,
combined with the Schauder theory for elliptic problems and the maximum principle.  
\end{abstract}

\thispagestyle{plain}
\markboth{VERA MIKYOUNG HUR}{STOKES WAVES WITH VORTICITY}

\section{Introduction}\label{S:intro}

The problem of {\em surface water waves}, in its simplest form, concerns 
the two-dimensional dynamics of an incompressible inviscid fluid of infinite depth and 
the wave motion on its surface layer, acted upon by gravity. 
The effect of surface tension is neglected. 
Suppose for definiteness that in the $(x,y)$-Cartesian coordinates 
gravity acts in the negative $y$-direction and that the fluid at time $t$ occupies the region 
bounded above by the moving surface, given as the graph $y=\eta (t,x)$. 
In the fluid region %$\{(x,y):-\infty<y<\eta (t,x)\}$, 
the velocity field $(u(t,x,y),v(t,x,y))$ and the pressure $P(t,x,y)$ satisfy the Euler equations
\begin{equation}\label{E:euler}
\begin{aligned}
u_x+v_y=0, \quad& && u_t+uu_x+vu_y=-P_x, \\ 
& && v_t\,+uv_x+vv_y\,=-P_y-g,
\end{aligned}
\end{equation}
where $g>0$ denotes the acceleration due to gravity. Throughout, subscripts denote partial derivatives.
The flow is allowed to be {\em rotational}, characterized by the vorticity $\omega =v_{x}-u_{y}$. 
The kinematic and dynamic boundary conditions 
\begin{equation}\label{E:top}
v=\eta_t +u\eta_x \quad \text{and}\quad P=P_{atm}
\end{equation}
at the surface layer express, respectively, 
that the surface moves with the velocity of the fluid particles at the surface and that 
the pressure at the surface is constant atmospheric, denoted by $P_{atm}$. 
The boundary condition at the infinite bottom  
\begin{equation}\label{E:bottom}
(u,v) \to (0,0) \qquad \text{as $y \to -\infty$}
\end{equation}
states that the flow at great depths is practically at rest.

It is a matter of common experience that waves which may be typically observed 
on the surface of the ocean or the river are approximately periodic and
propagating of permanent form at a constant speed. 
Waves of this kind are referred to as {\em Stokes waves}.
Intuitively, they are symmetric waves whose profile rises and falls exactly once per wavelength.

Waves of Stokes' kind are among the few exact solutions of the water-wave problem\footnote{
The water-wave problem is greatly complicated by the nonlinearity at the moving surface, and
the existence of arbitrarily-shaped solutions as well as special solutions are not well understood.}
\eqref{E:euler}-\eqref{E:bottom}. 
To clarify their existence is therefore a fundamental mathematical issue, 
and it is the subject of investigation here. 
Furthermore, they are a genuine nonlinear phenomenon. 
In his formal yet far-reaching consideration \cite{Sto80}, Stokes observed that 
characteristics of actual water waves deviate significantly from what the linear theory predicts. 
In particular, he conjectured that a periodic wave of maximum height exists and
%an ``extreme" wave, the tallest one among periodic waves, 
it is distinguished by a sharp cusp at the wave crest with the contained angle of 120 degree\footnote{
The linear theory gives 90 degree, in contrast.}.

%Existence of Stokes waves of all amplitude, therefore, is important 
%in the stability analysis and for other purposes, which is under investigation here. 

When $\omega=0$, namely, in the {\em irrotational} setting, 
the existence theory of Stokes waves dates back to the construction of small-amplitude waves 
by Levi-Civita \cite{LC25} and independently by Nekrasov \cite{Nek51},
and it includes the global theory by Krasovskii \cite{Kra61} and by Keady and Norbury \cite{KeNo78}.
Stokes' conjecture is proved in the works of 
Amick, Fraenkel and Toland \cite{AFT, Tol} and McLeod \cite{McL}. 
These results are based on the reformulation of the problem as Nekrasov's integral equation
and they are reviewed in \cite{Tol96}.
Further advances are made in \cite{BDT00a, BDT00b, BuTo03}
based on the formulation of the problem as Babenko's pseudo-differential equation \cite{Bab87}.

While the zero-vorticity setting may serve as an approximation under certain circumstances 
and it dominates in the existing literature,
ocean currents typically carry vorticities in their wind-drifted boundary layers \cite{Mei}. 
Moreover, the governing equations of water waves allow for rotational motions. 
Gerstner \cite{Ger02} in 1809 (earlier than Stokes!) found an explicit formula 
for periodic traveling waves on deep water with a particular nonzero vorticity; 
in the irrotational setting, no traveling-wave solution of the water-wave problem 
\eqref{E:euler}-\eqref{E:bottom} is known in the closed form. 
Perhaps, more striking are periodic traveling waves over a current 
with a critical layer and with a closed streamline \cite{Wah09, CoVa}; 
such waves cannot exist in the irrotational setting. Vorticity, in addition, 
has subtle influence on the hydrodynamic stability \cite{HuLi08} of traveling water waves. 

The existence theory of irrotational Stokes waves is so quite complete that 
it is featured in a textbook \cite{BuTo03}, but their siblings, with vorticity, enjoy a different reputation.
Dubreil-Jacotin \cite{DJ34} used a partial hodograph transform 
to reformulate the original free boundary problem in a fixed domain
and she addressed the existence of small-amplitude waves by power series methods. 
Zeidler \cite{Zei73} later suggested to use a quasi-conformal mapping and 
treated small-amplitude waves in the finite-depth case as well as under gravity and surface tension. 
The global theory began only recently when Constantin and Strauss \cite{CoSt04} 
recognized that Dubreil-Jacotin's formulation could be regarded as 
an abstract operator equation in a Banach space. They employed the topological degree theory, 
as adapted by Healey and Simpson \cite{HeSi98} for a general class of nonlinear elliptic operators,
and the global bifurcation theory of Rabinowitz \cite{Rab71}, and in the finite-depth case
they obtained a global connected set of solutions for a general class of vorticities. 
In the infinite-depth case, the author \cite{Hur06} extended the methods in \cite{CoSt04}
but when the vorticity is small, non-negative and monotone with depth. 
As a matter of fact, the vorticity function of Gerstner's trochoidal waves \cite{Ger02} is not treated. 
The present purpose is to establish  for a general class of vorticities
a global bifurcation result for Stokes waves on deep water.
The main result (Theorem \ref{T:main}) does not have any restriction on the sign of the vorticity.
%Furthermore, it improves the previous result in \cite{Hur06} for non-negative and monotone vorticities.

Following \cite{CoSt04}, there have been vigorous activities 
in studies of traveling water waves with vorticity.
In \cite{KoSt08a, KoSt08b}, waves found in \cite{CoSt04} are numerically computed. 
For arbitrary vorticities, small-amplitude solitary waves over channels of finite depth 
are constructed in \cite{Hur08a} and independently in \cite{GrWa08}. 
For arbitrary vorticities, symmetry property is studied 
in \cite{Hur07, CEW} for periodic waves of finite depth and in \cite{Hur08b} for solitary waves. 
Partial results regarding the Stokes conjecture on a ``limiting" wave are given in \cite{Var08, Var09}. 
Hydrodynamic stability of periodic waves of finite depth is studied in \cite{HuLi08}.

\

\noindent{\bf The main result}
Theorem \ref{T:main} states that for a general class of vorticities 
a global connected set of nontrivial Stokes-wave solutions exists and that 
the continuum contains a sequence of solutions for which 
either the speed of wave propagation becomes arbitrarily large (cavitation)
or the relative flow speed somewhere in the fluid region becomes arbitrarily close to zero (stagnation). 
Furthermore, Theorem \ref{T:main-} states that if the vorticity is non-positive and monotone with depth
then stagnation can only  occur at the wave crest.
If the vorticity is non-negative and monotone with depth, Theorem \ref{T:main+} states that
in the event of stagnation, it occurs either at the infinite bottom or on the free surface. 
For non-negative and monotone vorticities, 
if the relative flow speed is bounded along the continuum and if vorticity is, in addition, sufficiently small  
then stagnation occurs either at the infinite bottom or at the wave crest. 

The class of vorticity functions admissible in Theorem~\ref{T:main},
extending considerably that in the earlier work \cite{Hur06},
does include the vorticity distribution of Gerstner's trochoidal waves \cite{Ger02}.
Furthermore, for small, non-negative and monotone vorticities, 
Theorem \ref{T:main+} improves the result in \cite{Hur06} in that 
the smallness condition of the vorticity \eqref{C:small} is more straightforward than that in \cite{Hur06}.

Cavitation is considered to be unphysical. In fact, in the irrotational setting \cite{Tol96, BuTo03} 
as well as in the finite-depth case with vorticity \cite{CoSt04}, 
speed of wave propagation is a priori bounded and hence 
a ``limiting" wave along the continuum must exhibit stagnation. 
An important open problem is to obtain bounds for speed of wave propagation
in the infinite-depth case and with vorticity. 
In the irrotational setting, a bound for traveling speed \cite{Tol96, BuTo03} is established
by studying the kernel associated to Nekrasov's integral equation.
Unfortunately, such an integral representation of solutions \cite{Nek51, Tol96} 
critically hinges upon irrotationality, and presently no analog in the rotational setting is available. 
Section \ref{S:Zei73} presents the reformulation of the problem with vorticity,
which is potentially useful in obtaining an integral representation of solutions.

Another important open problem is to show 
for non-positive vorticities that stagnation occurs at the wave crest,
i.e., to remove the monotonicity assumption from Theorem \ref{T:main-}. 
A difficulty lies in that the relative flow speed does not possess a maximum principle. 
In the finite-depth case, the result is established in \cite{CoSt07, Var08}
by studying the maximum of the relative horizontal speed along the free surface, 
along the bottom and below the wave crest.

\

\noindent{\bf Ideas of the proof}
The present treatment is influenced by \cite{CoSt04, Hur06}, 
based on the reformulation via a partial hodograph transform of 
the original semilinear elliptic boundary value problem with free surface (Section \ref{SS:formulation}) 
as a quasilinear elliptic boundary value problem in a fixed semi-infinite strip 
(Section \ref{SS:reformulation}). 
In the finite-depth case \cite{CoSt04}, the Fredholm property of the operator describing the problem 
follows in the standard manner by the Schauder theory for elliptic problems 
and the embedding properties of H\"older spaces of functions in a bounded domain.
Consequently, a global connected set of solutions is obtained as an application of 
a generalized degree theory \cite{HeSi98} and global bifurcation theory \cite{Rab71}. 
In the infinite-depth case, unfortunately, the unboundedness of the domain prevents 
the operator from being Fredholm,
and thus a degree-theoretic argument is not directly applicable; see Section \ref{SS:methods}. 
It is noteworthy that in the irrotational setting (of infinite-depth) \cite{Tol96, BuTo03}, 
the problem further reduces to an equation for a quantity 
defined at the one-dimensional free surface, and specialized theory of bifurcation applies.

In order to overcome the failure of Fredholm property, 
as is done in \cite{Hur06}, the operator is approximated by a sequence of Fredholm operators. 
The framework of a generalized degree theory and global bifurcation theory 
then applies to each approximate problem, and 
a global connected set of its nontrivial solutions is constructed;
see Section~\ref{S:degree}, Section~\ref{SS:local} and Section~\ref{SS:global-approx}. 
The heart of the matter of the proof is in Section \ref{SS:global}
to take the limit of the continua of approximate solutions 
and to show that the limit set of nontrivial solutions of the original problem 
is connected and unbounded in the Banach space in use.
The connectedness of the limit set necessitates, in light of Wyburn's lemma 
(\cite{Why}, see also Theorem~\ref{T:A6}), 
a uniform decay at infinity of nontrivial solutions of the original problem. 
In the earlier work \cite{Hur06}, {\em ad hoc} arguments enforce that
for a restricted class of vorticities solutions decay exponentially at infinity. 
Here, an exponential decay of solutions is established in Lemma~\ref{L:exp-decay} 
by consideration of a Phragm\'en-Lindel\"of theorem, 
and it does not require any special property of the vorticity.
Then, in Lemma \ref{L:nodal} robust nodal properties along the continuum assert that 
the continuum of nontrivial solutions of the original problem is unbounded.

%%%%%%%%%%%%%%%%%%%%%%%%%%%%%%%%%%%%%%%%%%%%%%%%%%%%%%%%%%%%%%%%%%%%%%%%%%%%%%%%%%%%%%%%%%%%%%%%%%%%%%%%%%%%%%%%%%%%%%%%%%%%%%%%%%%%%%%%%%%%%%%%%%%%%%%%%
\section{Formulation and the main result}\label{S:formulation}

A detailed account is given of the passage from the traveling-wave problem 
of \eqref{E:euler}-\eqref{E:bottom} to an abstract operator equation in a Banach space.
The main results are stated. 
The failure of the Fredholm property of the operator is discussed, 
and approximate problems are designed.

%%%%%%%%%%%%%%%%%%%%%%%%%%%%%%%%%%%%%%%%%%%%%%%%%%%%%%%%%%%%%%%%%%%%%%%%%%%%%%%%%%%%%%%%%%%%%%%%%%%%%%
\subsection{The vorticity-stream function formulation}\label{SS:formulation}

The traveling-wave problem of \eqref{E:euler}-\eqref{E:bottom} seeks for a solution,
for which the wave profile, the velocity field and the pressure 
have the space-time dependence $(x-c t, y)$, where $c>0$ is the speed of wave propagation. 
In the frame of reference moving with the speed~$c$, 
the wave profile and the flow underneath it appear to be stationary. Let 
\[ \Omega_\eta = \{ (x,y): -\infty<x<\infty, \, -\infty<y<\eta(x)\}, \quad S_\eta=\{ (x,\eta(x)): -\infty<x<\infty\}\]
denote, respectively, the (stationary) fluid domain and the free surface.

In studies of traveling water waves, it is customary to introduce the (relative) stream function $\psi(x,y)$,
defined $\overline{\Omega}_\eta$ as
\begin{equation} \label{D:stream}
\psi _{x}=-v,\qquad \psi _{y}=u-c 
\end{equation}%
and $\psi (0,\eta (0))=0$. 
Accordingly, we formulate the traveling-wave problem of \eqref{E:euler}-\eqref{E:bottom} as
the free boundary problem as: 

for a function $\gamma(r)$ defined for $r \in [0,\infty)$ and for a parameter $c>0$, 
find a curve $y=\eta(x)$ defined for $x \in \mathbb{R}$ 
and a function $\psi(x,y)$ defined in $\overline{\Omega}_\eta$ such that 
\begin{subequations}\label{E:stream}
\begin{alignat}{2}
\psi_y<&0\qquad  & &\text{in}\quad \overline{\Omega}_\eta{,}\label{E:no-stag} \\
\intertext{and} 
-\Delta\psi  =& \gamma(\psi) \qquad & & \text{in}\quad \Omega_\eta{,} \\
\psi  =& 0\quad & & \text{on}\quad S_\eta{,} \\
|\nabla\psi|^{2}+  2&gy=0 \qquad & & \text{on}\quad S_\eta{,} \\
\nabla \psi \to (&0,-c)\quad & & \text{as\quad $y \to -\infty$ uniformly for $x$}.
\end{alignat}
\end{subequations}
The derivation of \eqref{E:stream} is detailed in \cite[Section 2]{Hur06}. 

The condition \eqref{E:no-stag} means that no stagnation point\footnote{
By a stagnation point we mean a point where $\psi_y=0$. 
It is a slight abuse of terminology, since traditionally $|\nabla \psi|=0$ at a stagnation point.} 
exists in the fluid region. 
Field observations \cite{Lig78} as well as laboratory experiments \cite{ThKl97}
indicate that for wave patterns which are not near the spilling or breaking state, 
the speed of wave propagation is in general considerably larger 
than the horizontal velocity of any water particle.

The no-stagnation condition \eqref{E:no-stag} guarantees that \cite{Hur06}
the vorticity is globally a function of the stream function, denoted by $\omega =\gamma (\psi)$. 
It is reasonable to require that $\gamma(r) \to 0$ as $r \to \infty$. Furthermore, the function
\begin{equation}
\Gamma(p)=\int_0^p \gamma(-p')dp' 
\end{equation}
is required to be bounded for $-\infty<p\leq 0$. Let
\[
\Gamma_{\inf}= \inf_{-\infty<p\leq0}\Gamma(p),\qquad
\Gamma_{\infty}=\int^{-\infty}_0 \gamma(-p)dp.
\]

If the vorticity is non-negative and monotone with depth, i.e.
if $\gamma(r) \geq 0$ and $\gamma'(r) \leq 0$ for $r \in [0,\infty)$,
then \eqref{E:no-stag} is redundant. 
Indeed, by the maximum principle and the Hopf boundary lemma, 
any solution $\psi$ of (\ref{E:stream}b) subject to (\ref{E:stream}e) must acquire \eqref{E:no-stag}.

The boundary condition (\ref{E:stream}c) means that the free surface itself makes a streamline,
while (\ref{E:stream}d) is a manifestation of Bernoulli's law 
which states that the quantity $|\nabla \psi|^2+2gy$ is a constant on the free surface. 
The Bernoulli's constant only serves to relocate the origin in the $y$-direction 
and by adding an arbitrary constant to (\ref{E:stream}d) changes 
neither the free surface nor the velocity distribution in the fluid region.
Thus, without loss of generality, the constant is taken to be zero.
The hydrostatic pressure in the fluid region is given by 
\begin{equation}\label{E:P}
P(x,y)=P_{atm}-\frac{1}{2}|\nabla \psi (x,y)|^{2}-gy+\Gamma(-\psi(x,y)).
\end{equation}

In view to the {\em Stokes wave problem}, 
\eqref{E:stream} is further supplemented with the periodicity and symmetry conditions 
that $\eta(x)$ and $\psi(x,y)$ are even and $2L$-periodic in the $x$-variable, 
where $2L>0$ is the wavelength.

In this setting, $L$ and $c$ are considered as parameters whose values form part of the solution.
The wavelength $L$, in existence theory, is independent of other parameters, 
and hence it is held fixed in the sequel. 
On the other hand, the speed of wave propagation $c$ serves as the bifurcation parameter 
and it varies along a solution continuum.

%A distinctive feature of \eqref{E:stream} is that the wave profile $y=\eta(x)$ is not known a priori. 
%In other words, it is a {\em free boundary problem}.
In case $\gamma= 0$, namely in the irrotational setting, 
(\ref{E:stream}b) reduces to the Laplace equation 
and the nonlinearity of the problem resides only at the free surface. 
A nontrivial vorticity, in stark contrast, introduces additional nonlinearity 
in the field equation (\ref{E:stream}b), and it significantly complicates analysis.

%%%%%%%%%%%%%%%%%%%%%%%%%%%%%%%%%%%%%%%%%%%%%%%%%%%%%%%%%%%%%%%%%%%%%%%%%%%%%%%%%%%%%%%%%%%%%%%%%%%%%%
\subsection{The main results}\label{SS:results}

For a nonnegative integer $k$ and for $\alpha \in (0,1)$,
a domain $\Omega$ in $\mathbb{R}^2$ is called a $C^{k+\alpha}$ domain
if each point on its boundary, denoted by $\partial \Omega$, has a neighborhood in which 
$\partial \Omega$ is the graph of a $C^{k+\alpha}$ function.
Given a $C^{k+\alpha}$ domain $\Omega$ in the $(x,y)$-plane (not necessarily bounded), we define
\begin{equation}
C^{k+\alpha}_{per}(\overline{\Omega})=\{\,f \in C^{k+\alpha}(\overline{\Omega}): 
\text{$f$ is even and $2L$-periodic in the $x$-variable}\,\},
\end{equation}
where $C^{k+\alpha}(\overline{\Omega})$ is a H\"{o}lder space under the norm 
\begin{multline*}
\|f\|_{C^{k+\alpha}(\overline{\Omega})}=
\sum^{k}_{k_1+k_2=0}\sup_{\overline{\Omega}}|\partial^{k_2}_y\partial^{k_1}_x f(x,y)|\\
+\sup_{k_1+k_2=k}\,\sup_{\substack{(x,y)\neq (x',y')\\ \overline{\Omega}}}
\frac{|\partial^{k_2}_y\partial^{k_1}_x f(x,y) - \partial^{k_2}_y\partial^{k_1}_x f(x',y')|}
{((x-x')^2+(y-y')^2)^{\alpha/2}}.
\end{multline*}
This notation is extended in an obvious way to the case 
when $\alpha=0$ and to functions of a single variable.

The main result of this article concerns the existence of 
nontrivial Stokes waves on deep water for a general class of vorticities.

\begin{theorem}\label{T:main}
Let $L>0$ be held fixed. 
Suppose that the vorticity function $\gamma \in C^{1+\alpha}([0,\infty))$, $\alpha \in (0,1)$, 
satisfies that $\gamma(r) \in O(r^{-2-2\rho})$ as $r \to \infty$ for some $\rho>0$ and that
\begin{equation}\label{C:bifurcation}
\int^0_{-\infty} \left( 2(2\Gamma(p)-2\Gamma_{\inf})^{3/2}
+\left(\frac{\pi}{L}\right)^2(2\Gamma(p)-2\Gamma_{\inf})^{1/2} \right) e^{2p}dp <g.
\end{equation}

There exists a connected set $\mathcal{C}$ of solution triples $(c, \eta, \psi)$ in the space 
$\mathbb{R}_+ \times C^{3+\alpha}_{per}(\mathbb{R})
\times C^{3+\alpha}_{per}(\overline{\Omega}_\eta)$
of the system {\rm (\ref{E:stream}b)-(\ref{E:stream}e)} such that

\renewcommand{\labelenumi}{{\rm(\roman{enumi})}}
\begin{enumerate}
\item $\mathcal{C}$ contains a trivial solution which corresponds to 
a horizontal shear flow $\psi_x= 0$ under the flat surface $\eta= 0$ and
\item there is a sequence of solution triples  $\{(c_j, \eta_j, \psi_j)\}$ in $\mathcal{C}$, for which
\[
\text{either\/}\quad
{\lim_{j \to \infty} c_j= \infty }
\quad \text{or\/} \quad
{\lim_{j\to \infty}\sup_{\overline{\Omega}_{\eta_j}} \partial_y \psi_j(x,y)} = 0.
\]
\end{enumerate}
Moreover, each nontrivial solution triple  $(c, \eta, \psi)$ in $\mathcal{C}$ enjoys the following properties:
\begin{enumerate}\addtocounter{enumi}{2}%
\item $\eta$ and $\psi$ are even and $2L$-periodic in the $x$-variable;
\item $\eta$ has a single maximum (crest) at $x=0$ and 
a single minimum (trough) at $x=\pm L$ per wavelength;
\item the wave profile is monotone from crest to trough, i.e., 
$\eta_x(x)>0$ for $-L<x<0$ and $\eta_x(x)<0$ for $0<x<L$;
\item the speed of wave propagation is larger than the horizontal particle velocity 
everywhere in the fluid region, i.e. $\psi_y<0$ in the fluid region; and 
\item $\psi_x(x,y)>0$ for $-L<x<0$ and $\psi_x(x,y)<0$ for $0<x<L$.
\end{enumerate}
\end{theorem}

The condition \eqref{C:bifurcation} ensures local bifurcation. 
A more general condition for local bifurcation is in \eqref{C:bifurcation-g}.
If $\Gamma$ is small, then \eqref{C:bifurcation} is valid.
%\[ 2\sqrt{2\Gamma_{\sup}}^{3}+\left(\frac{\pi}{L}\right)^2\sqrt{2\Gamma_{\sup}} \leq g.\]

Theorem \ref{T:main} presents two alternatives in~(ii).
If the first alternative realizes, 
the speed of wave propagation along the continuum becomes unboundedly large,
and correspondingly, the hydrostatic pressure becomes unboundedly low. 
This phenomenon is called {\em cavitation}. 
The second alternative means that the continuum contains waves 
whose relative flow speed somewhere in the fluid region becomes arbitrarily close to zero.
In other words, there is a region of almost stagnant fluid, a region carried along by the traveling wave. 
This phenomenon is called {\em stagnation}.

If the vorticity is monotone with depth, then the conclusion in (ii) can be refined. 

\begin{theorem}\label{T:main-}
Under the hypotheses of Theorem \ref{T:main}, 
if, in addition, $\gamma(r) \leq 0$ and $\gamma'(r) \geq 0$ for $0 \leq r<\infty$,
then the conclusions of Theorem \ref{T:main} holds with {\rm (ii)} replaced by 
\renewcommand{\labelenumi}{{\rm($\text{\roman{enumi}}'$)}}
\begin{enumerate}\addtocounter{enumi}{1}%
\item there is a sequence of solution triples $\{(c_j,\eta_j,\psi_j)\} \subset \mathcal{C}$, for which 
\[
\text{either}\quad{\lim_{j \to \infty} c_j= \infty}\quad
\text{or}\quad{\lim_{j\to \infty} \partial_y \psi_j(0,\eta_j0))} = 0.
\]
\end{enumerate}
\end{theorem}

\begin{theorem}\label{T:main+} 
Under the hypotheses of Theorem \ref{T:main}, 
if, in addition, $\gamma(r)\geq 0$ and $\gamma'(r) \leq 0$ for $0\leq r<\infty$, 
then the conclusions of Theorem \ref{T:main} holds with {\rm (ii)} replaced by 
\renewcommand{\labelenumi}{{\rm($\text{\roman{enumi}}''$)}}
\begin{enumerate}\addtocounter{enumi}{1}%
\item there is a sequence of solution triples $\{(c_j,\eta_j,\psi_j)\} \subset \mathcal{C}$, for which 
\[
\text{either}\quad{\lim_{j \to \infty} c_j= \infty}
\quad \text{or\/}\quad {\lim_{j \to \infty} c_j= 0}\quad \text{or}\quad
{\lim_{j\to \infty}\max_{0\leq x\leq L} \partial_y \psi_j(x,\eta_j(x))} = 0.
\]
\end{enumerate}

If $\psi_y(\pm L, \eta(\pm L))<M$ for all $(c,\eta,\psi) \in \mathcal{C}$.
and if, in addition, $\gamma(0)$ is sufficiently small so that 
\begin{equation}\label{C:small}
 g+\gamma(0)M \geq 0,
\end{equation}
then {\rm (ii$''$)} is further replaced by 
\renewcommand{\labelenumi}{{\rm($\text{\roman{enumi}}'''$)}}
\begin{enumerate}\addtocounter{enumi}{1}%
\item there is a sequence of solution triples $\{(c_j,\eta_j,\psi_j)\} \subset \mathcal{C}$, for which 
\[\text{either}\quad{\lim_{j \to \infty} c_j= 0}\quad \text{or\/}\quad 
\lim_{j\to \infty}\partial_y \psi_j(0,\eta_j(0)) = 0.\]
\end{enumerate}
\end{theorem}

Theorem \ref{T:main-} states that if the vorticity is non-positive and monotone with depth then 
stagnation, if occurs, must be at the wave crest. 
If the vorticity is non-negative and monotone with depth, Theorem \ref{T:main+} states that
when the second alternative in (ii) of Theorem \ref{T:main} realizes, 
stagnation occurs either at the infinite bottom or somewhere on the free surface. 
If  the relative flow speed at the wave trough is bounded along the continuum, 
and if, in addition, the vorticity is sufficiently small, then cavitation does not occur
and stagnant occurs either at the infinite bottorm or at the wave crest 
(see (ii$'''$) in Theorem \ref{T:main+}).

The conclusion of Theorem \ref{T:main} or Theorem \ref{T:main-}, in case of zero vorticity, 
partly recovers the well-known result (\cite{Tol96}, for instance) that 
the continuum of irrotational Stokes waves contains a ``limiting" wave with stagnation at the wave crest.
In other words, the second alternative in (ii) of Theorem \ref{T:main} 
or (ii') of Theorem \ref{T:main-} occurs.
Theorem \ref{T:main} or Theorem \ref{T:main-} also explains 
the existence of Gerstner's trochoidal waves \cite{Ger02}. 
Indeed, the vorticity function corresponding to Gerstner's waves \cite{OkSh01}, given by
\[ \gamma(\psi) =-2m^2\frac{e^{2b(\psi)}}{1-m^2e^{2b(\psi)}},\]
where $0\leq m<1$ and $-\infty<b<0$, is non-positive and monotone with depth.

Theorem \ref{T:main+} improves the result in \cite{Hur06}.
Cavitation, if occurs, is shown to be a consequence of that 
the speed of wave propagation becomes unboundedly large. 
Furthermore, \eqref{C:small} gives a more straightforward smallness condition of the vorticity
than that in \cite{Hur06}.

In the finite-depth case \cite{CoSt04}, 
the second alternative in (ii) of Theorem \ref{T:main} (stagnation) realizes. 
Instead of speed of wave propagation, in \cite{CoSt04} 
Bernoulli's constant serves as the bifurcation parameter, 
and the parameter values are shown to be subcritical; 
solitary water waves of small-amplitude bifurcate for supercritical values of parameter \cite{Hur08a}. 
While solitary water waves are not expected to exist in the infinite-depth case \cite{Cra02, Hur09b},
nevertheless, I conjecture that a limiting Stokes wave with vorticity exhibits stagnation.

In the irrotational setting \cite{Tol96, BuTo03} 
speed of wave propagation of Stokes waves on deep water is shown to be a priori bounded
by studying the kernel associated to Nekrasov's integral equation.
%Without such bounds, in the rotational setting
%unphysical limiting waves with cavitation cannot be eliminated.
It is noteworthy that even when existence theory is based on Babenko's pseudo-differential equation 
\cite{BuTo03}, a bound for traveling speed uses Nekrasov's integral equation.
Such an integral representation of solutions, unfortunately, critically depends on that 
the stream function is harmonic, and thus it is not readily available in the rotational setting.
Section \ref{S:Zei73} presents the reformulation of the problem via a quasi-conformal transform,
which has structural similarity to the formulation of the irrotational problem in \cite{LC25}, 
and thus it is potentially useful to obtaining an integral representation of solutions.

In the finite-depth case \cite{Var08, Var09}, 
if the vorticitiy is non-positive (not necessarily monotone with depth),
stagnation is shown to occur at the wave crest. 
The same result is expected to hold in the infinite-depth case, but no proof is given presently. 
Section \ref{S:properties} collects properties of Stokes waves of infinite depth
which are relevant to study the location of stagnation points.

%%%%%%%%%%%%%%%%%%%%%%%%%%%%%%%%%%%%%%%%%%%%%%%%%%%%%%%%%%%%%%%%%%%%%%%%%%%%%%%%%%%%%%%%%%%%%%%%%%%%%%
\subsection{Reformulation: reduction to an operator equation}\label{SS:reformulation}
Under the no-stagnation condition \eqref{E:no-stag}, 
exchanging the roles of the $y$-coordinate and $\psi$ offers a reformulation of 
(\ref{E:stream}b)-(\ref{E:stream}e) in a fixed domain, which serves as the basis of the existence theory. 

Let 
\[
q=x \quad\text{and}\quad p=-\psi(x,y)
\]
be new independent variables. 
They map the fluid region of one period $\{(x,y) \in \Omega_\eta: -L<x<L\}$ 
to the fixed semi-infinite strip $(-L,L)\times (-\infty,0)$ in the $(q,p)$-plane and 
the free surface of one period $\{(x, \eta(x)): -L<x<L\}$ to the top boundary $(-L,L)\times \{0\}$ of the strip.
Let 
\[
R=\{ (q,p) : -L< q<L, \, -\infty<p<0\}, \qquad T=\{ (q,0) : -L<q<L \}.
\]
Accordingly, the depth function \[h(q,p)=y\] replaces the dependent variable. 
It is straightforward to show that 
\begin{equation}\label{E:h-der}
h_q=-\frac{\psi_x}{\psi_y}, \qquad h_p= -\frac{1}{\psi_y}. 
\end{equation}

By the above {\em partial} hodograph transform, 
the semilinear elliptic free boundary problem \eqref{E:stream} is reformulated as 
the following quasilinear elliptic boundary value problem in the fixed domain $R$:
\begin{subequations}\label{h-prob}%
\begin{alignat}{2}
(1+h_q^2)h_{pp}-2h_ph_qh_{pq}&+h^2_ph_{qq}=
-\gamma\Pf h^3_p \qquad& &\text{in}\quad R{,} \tag{\ref{h-prob}a}\\
1+2ghh_p^2&+h^2_q=0 & &\text{on}\quad T{,} \tag{\ref{h-prob}b}\\
\nabla h=(h_q,h_p)&\to (0,1/c)\quad
\text{as }p \to -\infty& &\text{uniformly for $q$}{,}
\tag{\ref{h-prob}c}
\end{alignat}%
\end{subequations}%
where $h_p>0$ in $\overline{R}$ and $h$ is even and $2L$-periodic in the $q$-variable.

It is established in \cite[Lemma 3.1]{Hur06} that the above formulation is equivalent to \eqref{E:stream}.

A preliminary step of obtaining an operator equation for \eqref{h-prob}
is to identify its trivial solutions. 

\begin{lemma}[Trivial flows]\label{L:trivial}
Given $\gamma \in C^{1+\alpha}([0,\infty))$, $\alpha \in (0,1)$, 
for each $\lambda \in (-2\Gamma_{\inf}, \infty)$ the system \eqref{h-prob} has a solution
\begin{equation}\label{E:trivial}
h_{tr}(p)=h_{tr}(p;\lambda)=\int^p_0 (\lambda+2\Gamma(p'))^{-1/2}dp'-\frac{\lambda}{2g},
\end{equation}
which corresponds to the shear flow in the horizontal direction 
$$(\psi_{tr})_y(y)=-(\lambda+2\Gamma(p(y)))^{1/2}$$
under the flat surface $\eta(x)=0$, where $p(y)$ is an inverse of \eqref{E:trivial}.
\end{lemma}

%\begin{proof}
%These solutions do not depend on $q$, and thus (\ref{h-prob}a) reduces to $H''=-\gamma(-p)(H')^3$.
%Here and elsewhere the prime denotes differentiation with respect to $p$.
%Solutions to this ordinary differential equation are
%\[
%H'(p)=\frac{1}{\sqrt{\lambda+2\Gamma(p)}},
%\]
%which are defined for $\lambda  > -2\Gamma_{\inf} \geq 0$.
%The formula \eqref{E:trivial} then follows upon integrating the above and 
%using the boundary condition $1+2gH(0)(H'(0))^2=0$.

%The boundary condition at infinity, $H'(p) \to 1/c$ as $p \to -\infty$, relates 
%the bifurcation parameter $\lambda$ with the wave speed $c$ by $c^2=\lambda+2\Gamma_\infty$. 
%\end{proof}

The proof is in \cite[Lemma 3.2]{Hur06}, and it is omitted. 
%We only pause here to remark that 

In the bifurcation analysis below, instead of $c$ the square of the (relative) upstream speed 
$\lambda=(h_{tr}')^{-2}(0)=(\psi_{tr})^2(0)$ of the trivial flow \eqref{E:trivial} 
serves as the bifurcation parameter. 
For each $\lambda \in (-2\Gamma_{\inf}, \infty)$, the speed of wave propagation 
is determined by $\lambda$ by $c^2=\lambda+2\Gamma_\infty$. 

It is convenient to make use of the shorthand
\begin{equation}\label{D:a}
a(\lambda)=a(p;\lambda)=(\lambda+2\Gamma(p))^{1/2}.
\end{equation}
The derivatives of $h_{tr}$ can be expressed in terms of~$a$ as 
\[
h_{tr}'(p)=a^{-1}(p;\lambda), \qquad h_{tr}''(p)=-\gamma(-p)a^{-3}(p;\lambda).
\]
Note that $a(\lambda)$ is bounded for each $\lambda \in (-2\Gamma_{\inf}, \infty)$.

\

In order to tackle the existence question for solutions of \eqref{h-prob} via bifurcation theory,
we need to further reformulate the problem as 
an abstract operator equation in the form $F(\lambda,w)=0$,
where $w$ belongs to a Banach space. To this end, let 
\begin{equation}\label{D:w}
h(q,p)=h_{tr}(p)+w(q,p).
\end{equation}
%be the difference between the depth of a nontrivial flow and the depth of the underlying trivial flow. 
Then, $w(q,p) \to 0$ as $p \to -\infty$ uniformly for $q$.

We introduce the function spaces in use. Let 
\begin{align*}
X=\{\,&f \in C^{3+\alpha}_{per}(\overline{R}):
\partial^{k_2}_p\partial^{k_1}_q w \in o(1) \text{ as } p \to -\infty,\; 
k_1+k_2 \leq 3 \text{ uniformly for $q$}\,\}{,} \\
Y_1=\{\,&f \in C^{1+\alpha}_{per}(\overline{R}):
\partial^{k_2}_p\partial^{k_1}_q  w \in o(1) \text{ as } p \to -\infty, \; 
k_1+k_2\leq 1\text{ uniformly for $q$}\,\}{,}
\end{align*}
and $Y_2=C^{2+\alpha}_{per}(T)$. Recall that the subscript $per$ means 
evenness and $2L$-periodicity in the $q$-variable. Let $Y=Y_1 \times Y_2$ with the product topology.
We equip $X$ and $Y$ with the H\"{o}lder norms (thus rendering them Banach spaces):
\[
\|\cdot\|_X:= \|\cdot\|_{C^{3+\alpha}(\overline{R})}\,{,} \qquad\|\cdot\|_Y:= \|\cdot\|_{Y_1}+\|\cdot\|_{Y_2}\,{,}
\]
where $\|\cdot\|_{Y_1}=\|\cdot\|_{C^{1+\alpha}(\overline{R})}$
and $\|\cdot\|_{Y_2}=\|\cdot\|_{C^{2+\alpha}(T)}$.
Let $Z=C^0_{per}(\overline{R})$ have 
the usual maximum norm $\|\cdot \|_Z=\|\cdot\|_{C^0(\overline{R})}$.

The operator form of the Stokes wave problem is then given $\gamma(r)$ defined for $r \in [0,\infty)$ 
to find a nontrivial solution $(\lambda,w) \in \mathbb{R} \times X$ of
\begin{equation}\label{E:main}
F(\lambda,w)=0, 
\end{equation}
where 
\begin{equation}\label{D:F}
F(\lambda, w)=(F_1(\lambda,w),F_2(\lambda,w)):\mathbb{R} \times X \to Y,
\end{equation}
\begin{align}
\begin{split}
F_1(\lambda,w)=&(1+w^2_q)w_{pp} -2(a^{-1}\A+w_p)w_q w_{pq}+(a^{-1}\A+w_p)^2w_{qq}\\
&+\gamma\Pf(a^{-1}\A+w_p)^3-\gamma\Pf a^{-3}\A(1+w^2_q){,} \label{D:F_1}
\end{split}\\
F_2(\lambda,w)=&\left[1+(2gw-\lambda)(\lambda^{-1/2}+w_p)^2+w^2_q\right]_T.\label{D:F_2}
\end{align}  

%%%%%%%%%%%%%%%%%%%%%%%%%%%%%%%%%%%%%%%%%%%%%%%%%%%%%%%%%%%%%%%%%%%%%%%%%%%%%%%%%%%%%%%%%%%%%%%%%%%%%%
\subsection{Approximate problems}\label{SS:methods}
In the finite-depth case with vorticity \cite{CoSt04}
as well as in the irrotational setting (of infinite depth) \cite{Tol96, BuTo03},
the key to a successful existence theory ``in the large'' for Stokes waves lies in 
a generalized degree theory and global bifurcation theory. 

The rotational Stokes-wave problem in the finite-depth case in \cite{CoSt04} 
takes the same operator equation $F(\lambda, w)=0$ as in the infinite-depth case
(where $F$ is in \eqref{D:F_1} and \eqref{D:F_2})
but with the important difference that in the finite-depth case $w$ is considered in a finite rectangle,
whereas in the infinite-depth case it is considered in the semi-infinite strip $R$. 

In the finite-depth case, the equation $F(\lambda,w)=0$ 
gives an elliptic boundary value problem in the bounded domain, 
and the Fredholm property of $F$ follows \cite{CoSt04} 
from the Schauder theory for elliptic problems and the compact embeddings 
of H\"older spaces of functions in the bounded domain.
The existence of Stokes wave solutions in the finite-depth case \cite{CoSt04} then
uses the generalized degree theory, adapted by Healey and Simpson \cite{HeSi98}
for a general class of nonlinear Fredholm operators, and global bifurcation theory \cite{Rab71}.

In the infinite-depth case, unfortunately, a similar approach fails. 
Denoted by $F_w(\lambda,w)$ is the Fr\'{e}chet derivative of $F$ in the second argument 
at $(\lambda, w) \in \mathbb{R} \times X$. A straightforward calculation yields that 
\[
F_w(\lambda,w)=(A(\lambda,w), B(\lambda,w)),
\]
where
\begin{align}
\begin{split}\label{D:A}
A(\lambda,w)[\varphi]=&(1+w^2_q)\varphi_{pp}-
2(a^{-1}\A+w_p)w_q\varphi_{pq}+(a^{-1}\A+w_p)^2\varphi_{qq}  \\ 
&+\left(-2w_qw_{pq}+2(a^{-1}\A+w_p)w_{qq}+3\gamma\Pf(a^{-1}\A+w_p)^2\right)\varphi_p  \\
&+\left(2w_qw_{pp}-2(a^{-1}\A+w_p)w_{pq}-2\gamma\Pf a^{-3}\A w_q\right)\varphi_q{,}  
\end{split}\\
\begin{split}\label{D:B}
B(\lambda,w)[\varphi]=&\left[2(2gw-\lambda)(\lambda^{-1/2}+w_p)\varphi_p 
+ 2w_q\varphi_q+2g(\lambda^{-1/2}+w_p)^2\varphi \right]_{T}.
\end{split}
\end{align}
We shall show in Lemma \ref{L:fred} that 
the closed-ness of the range of $F_w(\lambda,w):X\to Y$ is equivalent to
the unique solvability of its ``limiting'' problem
\[ \varphi_{pp}+(\lambda+2\Gamma_\infty)^{-1}\varphi_{qq}=0\]
in the infinite strip $\{ (q,p): -L<q<L,\, -\infty<p<\infty \}$ in the $C^0_{per}$ class; see also \cite{VoVo03}. 
However, the spectrum of the operator $\partial_p^2+(\lambda+2\Gamma_\infty)^{-1}\partial_q^2$
defined in the infinite strip is $(-\infty,0]$, and the limiting problem has infinitely many solutions. 
In the infinite-depth case the operator $F$ defining the Stokes-wave problem is {\em not} Fredholm, 
and (generalized) degree theory may not be directly applicable.

This difficulty can be overcome by studying a sequence of ``approximate" problems 
\begin{equation}\label{E:approx}
F^\epsilon(\lambda,w)=0,
\end{equation} where $\epsilon>0$ and 
\begin{equation}\label{D:F^epsilon}
F^\epsilon(\lambda, w):=
\left(F_1(\lambda,w)-\epsilon w,F_2(\lambda,w)\right).
\end{equation}
Then, the Fredholm property of $F^\epsilon$ for each $\epsilon>0$ follows from that its limiting problem 
\[ \varphi_{pp}+(\lambda - 2\Gamma_\infty)^{-1} \varphi_{qq}-\epsilon \varphi=0\]
in the finite strip $\{(q,p): -L<q<L,\, -\infty<p<\infty\}$ admits only the trivial solution;
see Lemma \ref{L:fred} and Lemma \ref{L:proper}.
%Indeed, the operator  $\partial_p^2+(\lambda+2\Gamma_\infty)^{-1}\partial_q^2-\epsilon I$
%has only the continuous spectrum $(-\infty, -\epsilon]$, and thus
%the above limiting problem is uniquely solvable.

%%%%%%%%%%%%%%%%%%%%%%%%%%%%%%%%%%%%%%%%%%%%%%%%%%%%%%%%%%%%%%%%%%%%%%%%%%%%%%%%%%%%%%%%%%%%%%%%%%%%%%%%%%%%%%%%%%%%%%%%%%%%%%%%%%%%%%%%%%%%%%%%%%%%%%%%%
\section{Generalized degree for the approximating operators}\label{S:degree} 

For $\delta>0$ let us define the set
\begin{equation}\label{D:O_delta}
\OO= \left\{(\lambda,w)\in \mathbb{R} \times X : \lambda> -2\Gamma_{\inf} + \delta,\,
a^{-1}\A +w_p>\delta \text{ in } R,\, w<\frac{2\lambda-\delta}{4g} \text{ on $T$}\right\}.
\end{equation}
The purpose of this section is for each $\delta>0$ and for each $\epsilon>0$ to establish 
several properties of the operator $F^\epsilon$ on the set $\OO$ needed to define a topological degree.

%where $F^\epsilon$ is given in \eqref{D:F^epsilon}.
%In what follows, $\epsilon$ and $\delta$ are fixed positive constants.

First, for each $\delta>0$ and $\epsilon>0$ by continuity of $F^\epsilon$ it follows that 
$\OO$ is open in $\mathbb{R} \times X$. It is straightforward that
\[
F^\epsilon_w(\lambda,w):=(A^\epsilon(\lambda,w), B(\lambda, w)): X \to Y ,
\]
where $A^\epsilon(\lambda, w)=A(\lambda, w)-\epsilon w$
and $A(\lambda, w)$ and $B(\lambda, w)$ are given in \eqref{D:A} and \eqref{D:B},
%\begin{align}
%\begin{split}\label{D:A^epsilon}
%A^\epsilon(\lambda,w)[\varphi]=&
%(1+w^2_q)\varphi_{pp}-2(a^{-1}\A+w_p)w_q\varphi_{pq}+(a^{-1}\A+w_p)^2\varphi_{qq}\\
%&+\bigl(-2w_qw_{pq}+2(a^{-1}\A+w_p)w_{qq}+3\gamma\Pf(a^{-1}\A+w_p)^2\bigr)\varphi_p\\
%&+\bigl(2w_qw_{pp}-2(a^{-1}\A+w_p)w_{pq}-2\gamma\Pf a^{-3}\A w_q\bigr)\varphi_q-\epsilon\varphi{,}\end{split}\\
%B(\lambda,w)[\varphi]=& \left[2(2gw-\lambda)(\lambda^{-1/2}+w_p)\varphi_p 
%+ 2w_q\varphi_q+2g(\lambda^{-1/2}+w_p)^2\varphi \right]_{T}
%\label{D:B}
%\end{align}
%and that 
%\[
%F^\epsilon_{\lambda}(\lambda,w)=(F_{1\lambda}(\lambda,w), F_{2\lambda}(\lambda,w)),
%\]
%where 
%\begin{align*}
%F_{1\lambda}(\lambda,w)[\mu]=&
%\mu a^{-3}\A\Bigl(w_qw_{pq}-(a^{-1}\A+w_p)w_{qq}-\frac{3}{2}\gamma\Pf(a^{-1}\A+w_p)\\
%&\qquad\qquad+\tfrac{3}{2}\gamma\Pf a^{-2}\A(1+w_q^2)-3\epsilon(a^{-1}\A +w_p)^2w\Bigr){,} \\
%F_{2\lambda}(\lambda,w)[\mu]=&\left[
%-\mu(\lambda^{-1/2}+w_p)\left(\lambda^{-1/2}+w_p+(2gw-\lambda)\lambda^{3/2}\right)\right]_T.
%\end{align*}
%Upon inspection, it follows that $F^\epsilon_w(\lambda, w):X \to Y$ 
is continuous. %and hence,
Furthermore, $F^\epsilon:\mathbb{R} \times X \to Y$ is 
%Smoothness of $\gamma$ then ensures that $F^\epsilon:\mathbb{R} \times X \to Y$ is 
at least twice continuously Fr\'{e}chet differentiable.

Next, the principal parts of operators $A^\epsilon(\lambda,w)$ and $B(\lambda,w)$ are denoted by 
\begin{align}
A^P(\lambda,w)[\varphi]&
=(1+w_q^2)\varphi_{pp} -2(a^{-1}\A+w_p)w_q\varphi_{pq}
+(a^{-1}\A+w_p)^2\varphi_{qq}{,} \label{E:principal}\\
B^P(\lambda,w)[\varphi]&
=\left[2(2gw-\lambda)(\lambda^{-1/2}+w_p)\varphi_p+2w_q\varphi_q\right]_T{,} \label{E:principalb}
\end{align}
respectively. For each $\delta>0$ and for each $(\lambda,w) \in \OO$ 
note that the differential operator $A^\epsilon(\lambda, w)$ is uniformly elliptic 
with coefficient functions  bounded in $C^{2+\alpha}(\overline{R})$; 
the coefficients of the principal part satisfy
\[
4(1+w_q^2)(a^{-1}\A+w_p)^2 -4(a^{-1}\A+w_p)^2w_q^2 \geq 4\delta^2.
\]
Also, note that the boundary operator $B(\lambda,w)$ is uniformly oblique 
in the sense that it is bounded away from being tangential; 
the coefficient of $\varphi_p$ in $B(\lambda,w)$ satisfies
\[
|2(2gw-\lambda)(\lambda^{-1/2}+w_p)|>\delta^2 \qquad \text{on\quad$T$}.
\]
For each $\delta>0$ and for each $(\lambda,w) \in \OO$ thus 
$F^\epsilon_w(\lambda, w)=(A^\epsilon(\lambda,w),B(\lambda,w))$ 
satisfies the following Schauder estimate~\cite{ADN59}
\begin{equation}\label{E:Schauder2}
\|\varphi\|_X \leq C(\|A^\epsilon(\lambda,w)[\varphi]\|_{Y_1} 
+\|B(\lambda,w)[\varphi]\|_{Y_2} + \|\varphi\|_Z)
\end{equation}
for all $\varphi \in X$, where $C>0$ is independent of~$\varphi$.

Recorded in the next lemma is the Fredholm property of $F^\epsilon_w(\lambda,w)$.

\begin{lemma}[Fredholm property]\label{L:fred}
For each $\delta>0$ and for each $\epsilon>0$, for each $(\lambda, w) \in \mathcal{O}_\delta$  
the linear operator $F^\epsilon_w(\lambda,w)= (A^\epsilon(\lambda,w), B(\lambda,w)):X \to Y$
is a Fredholm operator of index zero.
\end{lemma}

\begin{proof}
The first step is to show that $F^\epsilon_w(\lambda,w):X \to Y$ is semi-Fredholm. 
That is, its range is closed in $Y$ and its kernel is finite-dimensional.

Let $\{\varphi_j\}$ be a bounded sequence in $X$. 
Let a sequence $\{ (y_{1j}, y_{2j})\}$ converge to $(y_1,y_2)$ in $Y$ as $j \to \infty$, and let 
\[ 
F^\epsilon_w(\lambda, w)[\varphi_j]= (A^\epsilon(\lambda,w), B(\lambda,w))[\varphi_j]=(y_{1j},y_{2j}),
\qquad j=1,2, \dots.\]
It is immediately that $\varphi_j \to \varphi$ in $C^3_{per}(\overline{R'})$ 
for some $\varphi$ for any bounded subset $R'$ of $R$. 
By continuity, then \[F^\epsilon_w(\lambda,w)[\varphi]=(y_1,y_2).\] 
Our goal is to show that $\varphi_j \to \varphi$ in $X$.

We claim that $\varphi_j \to \varphi$ in $C^0_{per}(\overline{R})$. 
Suppose, on the contrary, that there exists a sequence $\{(q_j,p_j)\}$ in $\overline{R}$ 
such that $p_j \to -\infty$ as $j \to \infty$, yet 
\begin{equation}\label{E:C0limit} 
|\varphi_j(q_j, p_j)-\varphi(q_j,p_j)| \geq \kappa >0\qquad \text{for all $j$}
\end{equation} for some $\kappa$. 
For each $j$, let us form the ``shifted difference'' 
\[ \vartheta_j(q,p)=\varphi_j(q,p+p_j)-\varphi(q,p+p_j)\] 
defined in the shifted domain $R_j:=\{(q,p) \in R: -\infty<p+p_j<0\}$. 
By construction, $\vartheta_j$ satisfies
\[A^\epsilon_j(\lambda, w(\cdot, \cdot +p_j))[\vartheta_j]=y_{1j}(\cdot, \cdot+p_j)-y_1(\cdot, \cdot+p_j)
\qquad \text{in $R_j$},\]
where the operator $A^\epsilon_j(\lambda,w)$ is obtained by shifting the coefficient functions 
of $A^\epsilon(\lambda,w)$ by $-p_j$ in the $p$-direction, i.e., 
by replacing $a(p;\lambda)$ by $a(p+p_j;\lambda)$ and $\gamma(-p)$ by $\gamma(-p-p_j)$; 
the value at $(q,p)$ of the function $w(\cdot, \cdot+p_j)$ is given by $w(q,p+p_j)$;
$y_{1j}(\cdot, \cdot+p_j)$, $y_1(\cdot, \cdot+p_j)$ are defined in the same manner.

Passing to the limit as $j \to \infty$ of the above, we obtain that
the (pointwise) ``limiting" function $\vartheta_0$ of $\vartheta_j$ is in the $C^0$ class
is defined in the ``limiting" domain \[R_0=\{ (q,p): -L<q<L,\, -\infty<p<\infty\}\] of $R_j$
and that it satisfies the ``limiting" equation
\begin{equation}\label{E:limiting}
(\vartheta_0)_{pp} +(\lambda+2\Gamma_\infty)^{-1}(\vartheta_0)_{qq}-\epsilon \vartheta_0=0
\qquad \text{in } R_0.
\end{equation}
The limiting equation is obtained by taking the (pointwise) limit of the coefficient functions of 
$A_j(\lambda, w(\cdot, \cdot +p_j))$ and $y_{1j}(\cdot, \cdot+p_j)$, $y_1(\cdot, \cdot+p_j)$ as $j \to \infty$
and it uses that $\nabla w(q,p+p_j), \nabla^2 w(q,p+p_j) \to 0$ as $j \to \infty$ 
for all $(q,p) \in \overline{R}$ and that $a(\lambda, p+p_j)=(\lambda+2\Gamma_{\infty})^{1/2}$ 
and $\gamma(-p-p_j) \to 0$ as $j \to \infty$ for all $-\infty<p\leq 0$.
Moreover, since $\varphi$ is even and $2L$-periodic in the $q$-variable, so is $\vartheta_0$.

It is standard that \eqref{E:limiting} admits only the trivial solution $\vartheta_0=0$.
Indeed, multiplying the equation by $\vartheta_0$ and integrating over $R_0$ yields that
\begin{equation}\label{E:energy} 
\iint_{R_0} \left((\vartheta_0)_{ p}^2+ (\lambda+2\Gamma_\infty)^{-1}(\vartheta_0)_q^2 
+\epsilon (\vartheta_0)^2\right)\,dqdp=0.
\end{equation}
This, however, contradicts \eqref{E:C0limit}, and thus proves the claim.

Since $A^\epsilon(\lambda,w)$ is uniformly elliptic with coefficient functions 
bounded in $C^{2+\alpha}(\overline{R})$ and since $B(\lambda,w)$ is uniformly oblique,
an application of the Schauder estimate \cite{ADN59} yields that 
\[ \|\varphi_j-\varphi\|_X \leq C(\|y_{1j}-y_1\|_{Y_1}+\|y_{2j}-y_2\|_{Y_2}+\|\varphi_j-\varphi\|_Z)\]
for all $j$, where $C>0$ is independent of $\varphi_j$ and $\varphi$. 
By the above claim, the last term of the right side vanishes as $j \to \infty$. 
Since the first two terms of the right side decreases to zero as $j \to \infty$ by hypothesis, 
it follows that $\varphi_j \to \varphi$ in $C^{3+\alpha}(\overline{R})$ as $j \to \infty$. 
That means, the range of $F^\epsilon_w(\lambda,w)$ is closed in $Y$. 

Repeating the above argument for $(y_{1j}, y_{2j})=(0,0)$ then yields that
the kernel of $F^\epsilon(\lambda,w)$ is a finite-dimensional subspace in $X$. 
Therefore, $F^\epsilon_w (\lambda,w)$ is semi-Fredholm.

The next step is to show that 
\begin{multline*}
F^\epsilon_w(\lambda,0)=\Big(\partial_p^2 +a^{-2}(\lambda)\partial_q^2
+3\gamma(-p)a^{-2}(\lambda)\partial_p-\epsilon I{,} 
\left[-2\lambda^{1/2}\partial_p +2g\lambda^{-1}I \right]_T\Big)
\end{multline*}
is Fredholm of index zero. Let 
\begin{align*}
L^\epsilon=\left(\partial_p^2 +(\lambda+2\Gamma_{\infty})^{-1}\partial_q^2-\epsilon I{,} 
\left[-2\lambda^{1/2}\partial_p +2g\lambda^{-1}I \right]_T\right)
\end{align*}
and let us consider the one-parameter family of operators
\[
(1-t)L^\epsilon+tF^\epsilon_w(\lambda,0):X \to Y\qquad \text{for}\quad t \in [0,1].
\] 
Note that $L^\epsilon$ is obtained by replacing the variable coefficients of $F^\epsilon(\lambda,0)$ 
by their pointwise limit as $p \to -\infty$. 
It is standard from the elliptic theory (see \cite[Chapter~3]{Kry96}, for instance) that 
$L^\epsilon:X \to Y$ is bijective. In particular, it is a Fredholm operator of index zero.
Since $F^\epsilon_w(\lambda,0)$ is semi-Fredholm from the previous step, 
it follows by the homotopy invariance of Fredholm index \cite[Chapter~4]{Kat67}
that $F^\epsilon_w(\lambda, 0)$ is also a Fredholm operator of index zero.

Finally, since $\mathcal{O}_\delta$ is connected, the assertion follows 
by the continuity of Fredholm index \cite[Chapter~4]{Kat67}.
\end{proof}

For our next preliminary result, we need several notations to describe. 
The domain of the operator $A^\epsilon(\lambda,w)$ is defined by
\begin{equation}\label{E:domain}
D(A^\epsilon(\lambda,w))=\{\,\phi \in X: B(\lambda,w)[\varphi]=0 \,\}.
\end{equation}
Note that $A^\epsilon(\lambda,w)$ restricted to $D(A^\epsilon(\lambda,w))$ is closed in $Y_1$.
The spectrum of $A^\epsilon(\lambda,w)$ is defined by
\begin{equation}\label{E:spectrum}
\sigma(\lambda,w) =\{ \mu \in \mathbb{C}:
A^\epsilon(\lambda,w)-\mu I: D(A^\epsilon(\lambda,w)) \to Y_1\text{ is not bijective}\},
\end{equation}
where $A^\epsilon(\lambda,w)$, $D(A^\epsilon(\lambda,w))$, and $Y_1$ 
are complexified in the natural way.
If $\mu \in \sigma(\lambda,w)$ and $\ker(A^\epsilon(\lambda,w)-\mu I)$ is nontrivial
then $\mu$ is called an {\em eigenvalue}. 
An eigenvalue $\mu$ is said to have finite algebraic multiplicity if 
\[\dim \ker(A^\epsilon(\lambda,w)-\mu I)^m=\dim \ker(A^\epsilon(\lambda,w)-\mu I)^{m+1}<\infty\]
for some positive integer $m$. In this case, 
$\dim \ker(A^\epsilon(\lambda,w)-\mu I)^m$ is called the algebraic multiplicity of $\mu$.

\begin{lemma}[Spectral properties]\label{L:spectral}
For each $\delta>0$ and $\epsilon>0$ and 
for each $(\lambda,w)\in\OO$ with\/ $|\lambda|+\|w\|_X \leq M$, where $M>0$, 
there exists a small constant $s>0$ and positive constants $C_1, C_2$ such that
\begin{equation}\label{E:spectral}
C_1\|\varphi\|_X \leq |\mu|^{\alpha/2}\|(A^\epsilon(\lambda,w)-\mu I)[\varphi]\|_{Y_1}
+|\mu|^{\alpha/2}\|B(\lambda,w)[\varphi]\|_{Y_2}+|\mu|^{(\alpha+1)/2} \|\varphi\|_{Z}
\end{equation}
for all $\varphi\in X$ and for all $\mu \in \mathbb{C}$ satisfying\/ $|{\rm arg}(\mu)| \leq \pi/2+s$
and $|\mu| \geq C_2\geq 1$ sufficiently large,
where $\alpha \in (0,1)$ is the H\"older exponent inherent from $X$ and $Y$.

Moreover, $\sigma(\lambda,w)$ possesses only finitely many eigenvalues 
in the sector $|arg(\mu)| \leq \pi/2+s$, each of which has a finite algebraic multiplicity.
The boundary operator $B(\lambda,w):X \to Y_2$ is surjective.
\end{lemma}

\renewcommand{\labelenumi}{{\rm(\arabic{enumi})}}

The proof is in \cite[Lemma 4.5]{Hur06}. See also \cite[Proposition 4.4]{HeSi98}.

Our last preliminary result is the properness of~$F^\epsilon$.

\begin{lemma}[Properness]\label{L:proper}
For each $\delta>0$ and $\epsilon>0$, the nonlinear operator $F^\epsilon$ is (locally) proper on $\Ob$,
i.e, $(F^\epsilon)^{-1}(K) \cap \overline{D}$ is compact in $\mathbb{R} \times X$ 
for each bounded set $D$ in $\Ob$ and for each compact set $K$ in $Y$.
\end{lemma}

\begin{proof}
Let $\{(\lambda_j,w_j)\}$ be a bounded sequence in $D \subset \Ob$.
Let $\{(y_{1j},y_{2j})\}$ be a convergent sequence in $K  \subset Y$. 
Let $(y_{1j}, y_{2j}) \to (y_1, y_2)$ as $j \to \infty$ and let 
\[
F^\epsilon(\lambda_j,w_j) =(y_{1j},y_{2j}), \qquad j=1,2,\dots.
\]
Our goal is find a subsequence of $\{(\lambda_j,w_j)\}$ which converges in $\mathbb{R} \times X$.

It is immediate that (possibly after relabling) 
$\lambda_j \to \lambda$ as $j \to \infty$ in $\mathbb{R}$ for some $\lambda$ 
and that $w_j \to w$ as $j \to \infty$ in $C^3_{per}(\overline{R'})$ for some $w$ 
for any bounded subset $R'$ of $\overline{R}$.
Moreover, by continuity, \[F^\epsilon(\lambda,w) = (y_1,y_2).\]

It is convenient to write $F^\epsilon$ in the operator form as 
\begin{subequations}\label{E:w-decomp1}
\begin{align}
F^\epsilon_1(\lambda,w)&=A^P(\lambda,w)[w] + f_1(\lambda,w) -\epsilon w{,} \\
F_2(\lambda,w)&=B^P(\lambda,w)[w] +f_2(\lambda,w).
\end{align}
\end{subequations}
Here, $A^P(\lambda, w)$ and $B^P(\lambda, w)$ are the principal parts 
of $A(\lambda,w)$ and $B(\lambda,w)$, respectively;
$f_1(\lambda,w)=\gamma(-p)(a^{-1}(\lambda)+w_p)^3-\gamma(-p)a^{-3}(\lambda)(1+w_q^2)$ and
$f_2(\lambda,w)=\lambda^{-1/2}(2gw-\lambda)(\lambda^{-1/2}+w_p)$.

The first step is to show that $w_j \to w$ in $C^0_{per}(\overline{R})$.
The proof is very similar to that in Lemma \ref{L:fred}. 
Suppose the convergence does not take place; there would be a sequence 
$\{(q_j,p_j)\}\subset\overline{R}$ such that $p_j \to -\infty$ as $j \to \infty$ but for some $\kappa$,
\begin{equation}\label{E:C0limit-1}
|w_j(q_j,p_j) - w(q_j,p_j)| \geq \kappa>0 \qquad \text{for all $j$}.
\end{equation}

For each $j$, as is done in Lemma~\ref{L:fred}, let us consider the function 
\[v_j(q,p) = w_j(q,p+p_j) -w(q,p+p_j),\]
defined in the domain $R_j=\{(q,p): -L<q<L,\,\, -\infty< p<-p_j\}$ (same as Lemma~\ref{L:fred}). 
It is straightforward to show that each $v_j$ satisfies 
\begin{align*} 
A^P_j(\lambda_j, w_j(\cdot, \cdot+&p_j) )[v_j] - \epsilon  v_j  \\
=&-\left(A^P_j(\lambda_j, w_j(\cdot, \cdot+p_j))
-A^P_j(\lambda, w(\cdot, \cdot+p_j))\right)w(\cdot, \cdot+p_j) \\
&-f_{1j}(\lambda_j, w_j(\cdot, \cdot+p_j))+f_{1j}(\lambda, w(\cdot,\cdot+p_j)) \\
&+\epsilon w(\cdot, \cdot+p_j)+y_{1j}(\cdot, \cdot+p_j)-y_1(\cdot, \cdot +p_j)
\end{align*}
in $R_j$. Here, $A^P_j(\lambda,w)$ is obtained 
by shifting the coefficient function $a(p;\lambda)$ of $A^P(\lambda, w)$ by $-p_j$ in the $p$-axis,
and likewise, $f_{1j}(\lambda, w)$ is obtained by shifting 
$a(p;\lambda)$ and $\gamma(-p)$ by $-p_j$ in the $p$-axis;
the value at (q,p) of the function $w_j(\cdot, \cdot+p_j)$ is given by $w_j(q,p+p_j)$
and $w(\cdot, \cdot+p_j)$, $y_{1j}(\cdot,\cdot+p_j)$, and $y_1(\cdot, \cdot+p_j)$ are defined similarly.

Passing to the limit as $j \to \infty$ of the above, similarly to the proof of Lemma~\ref{L:fred}, 
we conclude that there exist
the limiting function $v_0$ of $v_j$ in the $C^0_{per}$ class, 
the limiting domain $R_0=\{(q,p): -L<q<L,\, -\infty< p<\infty\}$ of $R_j$,
the limiting operator $\partial_p^2+(\lambda+2\Gamma_{\infty})^{-1}\partial_q^2-\epsilon I$ such that
\[
(v_0)_{pp}+(\lambda+2\Gamma_{\infty})^{-1}(v_0)_{qq}-\epsilon v_0=0\qquad \text{in } R_0.
\]
Indeed, $\lambda_j \to \lambda$ as $j \to \infty$ and 
$\nabla w(q,p+p_j) \to 0$ as $p_j \to -\infty$ for all $(q,p) \in \overline{R}$.
Since due to the limiting equation the energy integral \eqref{E:energy} of $v_0$ is zero,
it follows that $v_0=0$ in $R_0$. This, however, contradicts \eqref{E:C0limit-1} and proves  
the convergence of $\{w_j\}$ to~$w$ as $j \to \infty$ in $C^0_{per}(\overline{R})$.

Next, since $\{w_j\}$ is uniformly bounded under the $C^{3+\alpha}(\overline{R})$ norm, 
an interpolation inequality (see \cite[Lemma 6.32]{GiTr01} and  \cite[Theorem 3.2.1]{Kry96}) asserts that
$w_j \to w$ in $C^{k'+\alpha'}(\overline{R})$ for any $k'+\alpha' <3+\alpha$. 
Indeed, for any $s>0$ there exists a constant $C=C(s)>0$ such that 
\begin{equation}\label{E:interp}
\| w_j -w\|_{C^{k'+\alpha'}(\overline{R})} 
\leq  s \|w_j-w\|_{C^{3+\alpha}(\overline{R})}+C(s)\|w_j -w\|_{C^0(\overline{R})}.
\end{equation}

The final step is to employ the Schauder theory to obtain the convergence of $\{w_j\}$ 
in $C^{3+\alpha}(\overline{R})$. By virtue of the decomposition \eqref{E:w-decomp1}, 
the difference $w_j-w$ satisfies 
\begin{align*}
A^P(\lambda_j,w_j)[w_j-w]=y_{1j}-y_{1}-(A^P(\lambda_j,w_j)&-A^P(\lambda,w))[w] \\
&-(f_1(\lambda_j,w_j)-f_1(\lambda,w))+\epsilon (w_j-w)\\
B^P(\lambda_j,w_j)[w_j-w]=y_{2j}-y_{2}-(B^P(\lambda_j,w_j)&-B^P(\lambda,w))[w] \\
&-(f_2(\lambda_j,w_j)-f_2(\lambda,w)).
\end{align*}
Since $(\lambda_j,w_j) \in \Ob$, the Schauder estimates~\cite{ADN59} applies to
the operator $(A^P(\lambda_j, w_j), B^P(\lambda_j,w_j))$ to yield
\begin{equation}\label{E:Schauder4}
\| w_j-w\|_X \leq C(\|A^P(\lambda_j,w_j)[w_j-w]\|_{Y_1} +
\|B^P(\lambda_j,w_j)[w_j-w]\|_{Y_2} +\| w_j-w\|_Z).
\end{equation}

The result of the first step is that 
the last term on the right side of the above inequality tends zero as $j \to \infty$.
Since $\{w_j\} \subset X$ is bounded, by the interpolation inequality \eqref{E:interp} it follows that
coefficients of $A^P(\lambda_j, w_j)$ are equicontinuous in $j$. 
Since $\|w_j \|_{Y_1} \to \|w\|_{Y_1}$ as $j \to \infty$, moreover, it follows that
\[
\|(A^P(\lambda_j,w_j)-A^P(\lambda, w))[w]\|_{Y_1} \to 0 \qquad \text{as }j \to \infty.
\]
The convergence
\[
\| f_1(\lambda_j,w_j)-f_1(\lambda,w)\|_{Y_1} \to 0 \qquad \text{as }j \to \infty
\] 
follows by that $f_1$ consists of polynomial expressions of $w_p$ and $w_q$. 
These together with the convergence of $\{y_{1j}\}$ in $Y_1$ yield that 
\[
\| A^P(\lambda_j,w_j)[w_j-w]\|_{Y_1}\to 0 \qquad \text{as } j \to \infty.
\]

On the other hand, the standard Schauder estimates and the embedding properties
of H\"{o}lder spaces in the bounded domain $T$ confirm
\[
\|B^P(\lambda_j,w_j)[w_j-w]\|_{Y_2} \to 0 \qquad \text{as }j \to \infty.
\]
By \eqref{E:Schauder4}, therefore, $w_j \to w$  in $C^{3+\alpha}(\overline{R})$.
The assertion then follows since $X$ is a closed subspace of $C^{3+\alpha}(\overline{R})$.
\end{proof}

With the properties of $F^\epsilon$ established above in hand, 
for each $\delta>0$ and for each $\epsilon>0$ we define 
a generalization of the Leray-Schauder degree due to Healey and Simpson \cite{HeSi98} 
for $F^\epsilon(\lambda,w)$, where $(\lambda,w) \in \OO$. 
The detailed development is in \cite[Section 4]{HeSi98}.
Our interest in degree theory lies in that the degree is invariant under homotopy
and hence it can be used in global bifurcation theory.

%%%%%%%%%%%%%%%%%%%%%%%%%%%%%%%%%%%%%%%%%%%%%%%%%%%%%%%%%%%%%%%%%%%%%%%%%%%%%%%%%%%%%%%%%%%%%%%%%%%%%%%%%%%%%%%%%%%%%%%%%%%%%%%%%%%%%%%%%%%%%%%%%%%%%%%%%
\section{Existence theory for rotational Stokes waves}\label{S:existence}
%%%%%%%%%%%%%%%%%%%%%%%%%%%%%%%%%%%%%%%%%%%%%%%%%%%%%%%%%%%%%%%%%%%%%%%%%%%%%%%%%%%%%%%%%%%%%%%%%%%%%%%%%%%%%%%%%%%%%%%%%%%%%%%%%%%%%%%%%%%%%%%%%%%%%%%%%

Undertaken is the study of global bifurcation for \eqref{E:main}.
For each $0<\epsilon<1$, the existence of nontrivial solutions of \eqref{E:approx} is established
in a neighborhood of the trivial solution for a parameter value $\lambda^\epsilon$. 
Then, for each $0<\epsilon<1$ the local curve of solutions extends 
to a global connected set of solutions of \eqref{E:approx}. 
Finally, a global existence theory for \eqref{E:main} is obtained 
via abstract bifurcation theory and Wyburn's lemma in topology.

%%%%%%%%%%%%%%%%%%%%%%%%%%%%%%%%%%%%%%%%%%%%%%%%%%%%%%%%%%%%%%%%%%%%%%%%%%%%%%%%%%%%%%%%%%%%%%%%%%%%%%
\subsection{Local bifurcation for approximate problems}\label{SS:local}
It is readily seen that for each $\epsilon \geq 0$, the points $(\lambda,0)$, 
where $\lambda \in (-2\Gamma_{\inf}, \infty)$, 
form the line of trivial solutions in $\mathbb{R} \times X$.
The linearization of $F^\epsilon(\lambda,w)$ about the trivial solution $(\lambda,0)$ is
$F_w^\epsilon(\lambda, 0)=(A^\epsilon(\lambda,0),B(\lambda,0))$, where
\begin{align*}
A^\epsilon(\lambda,0)[\varphi]&= a^{-3}\A(a^3\A\varphi_p)_p+a^{-2}\A\varphi_{qq}-\epsilon \varphi{,} \\
B(\lambda,0)[\varphi]&=\left[-2\lambda^{1/2}\varphi_p+2g\lambda^{-1}\varphi\right]_T.
\end{align*}
Since $F^\epsilon$ is continuously Fr\'echet differentiable 
and $F^\epsilon_w(\lambda,0)$ is Fredholm of index zero, 
a necessary condition for bifurcation from a trivial solution $(\lambda,0)$ 
is that $F^\epsilon_w(\lambda,0):X \to Y$ is not injective, or equivalently, 
the boundary value problem of the self-adjoint equation
\begin{subequations}\label{E:eigen}%
\begin{alignat}{2}
(a^3\A\varphi_p)_p+(a\A&\varphi_q)_q-\epsilon a^3(\lambda) \varphi=0&  &\quad
\text{in \, $R$}{,} \\
\lambda^{3/2}&\varphi_p=g\varphi& & \quad \text{on \, $T$}
\end{alignat}%
\end{subequations}%
admits a nontrivial solution in~$X$.

\begin{lemma}[Bifurcation points]\label{L:eigen}
Suppose that $\gamma \in C^{1+\alpha}([0,\infty))$, $\alpha \in (0,1)$, satisfies \eqref{C:bifurcation}. 

\renewcommand{\labelenumi}{{\rm(\roman{enumi})}}
\begin{enumerate}
\item For each $0\leq \epsilon< 1$ there exist 
a unique $\lambda^\epsilon \in (-2\Gamma_{\inf}, gL/\pi-2\Gamma_{\inf}]$ 
and a unique (up to constant multiple) nontrivial solution $\varphi^\epsilon \in X$ to \eqref{E:eigen}.

\item For each $0\leq \epsilon< 1$, $\lambda^0\leq \lambda^\epsilon$ 
and $\lambda^\epsilon \to \lambda^0$ as $\epsilon\to 0+$. 
\end{enumerate}
\end{lemma}

\begin{proof}
(i) Let $0\leq \epsilon <1$ be held fixed. 
In view of evenness and periodicity of $\varphi \in X$ in the $q$-variable, 
we look for a solution of the form $\varphi^\epsilon(q,p)=\eP(p)\cos k(\pi/L)p$, 
where $k \geq 0$ is an integer. Then, $\eP$ solves the ordinary differential equation
\[(a^3\A v ')' - k^2(\pi/L)^2 a\A v- \epsilon a^3(\lambda)v=0.\]

Let us consider for $\lambda \in (-2\Gamma_{\inf}, \infty)$  the (singular) Sturm-Liouville problem
\begin{equation}\label{E:Sturm}
\begin{cases}
L^\epsilon v:= - (a^3\A v')' +\epsilon a^3\A v=\mu(\lambda) a(\lambda) v
\qquad \text{for }-\infty<p< 0{,} \\
\lambda^{3/2}v'(0)=gv(0) \end{cases}
\end{equation}
subject to the boundary conditions $v,v' \to 0$ as $p \to -\infty$.
Here, the prime denotes differentiation with respect to the $p$-variable.
Our aim is to find a $\lambda^\epsilon$ such that 
$\mu(\lambda^\epsilon)=-k^2(\pi/L)^2$ is a generalized eigenvalue 
of \eqref{E:Sturm} (such that $L^\epsilon v= \mu a(\el) v$ for some $v \not\equiv 0$), 
where $k\geq 0$ is an integer.

The proof is based on the variational consideration of \eqref{E:Sturm}. 
Let us define the Rayleigh quotient 
\begin{align}\label{D:R}
R^\epsilon(\lambda)=R^\epsilon(v;\lambda)= \frac{-gv^2(0) 
+ \int^0_{-\infty} a^3\A(v')^2dp+\epsilon\int^0_{-\infty}a^3\A v^2dp\,}{\int^0_{-\infty} a\A v^2dp}
\end{align}
and consider the minimization problem
\begin{equation}
\Lambda^\epsilon(\lambda) = \inf\{\,R^\epsilon(v;\lambda):
v \in H^1((-\infty,0))\text{ and }v \not\equiv 0\, \},\label{D:Lambda}
\end{equation}
where $H^1((-\infty,0))$ denotes the $L^2$-Sobolev space in the $p$-variable.
It is well known that $L^\epsilon$ with the boundary conditions in \eqref{E:Sturm} 
has the continuous spectrum $[\epsilon, \infty)$.
%A straightforward calculation shows that $R^\epsilon(\lambda)$ is bounded from below for each $\lambda$. More precisely, $R^\epsilon(\lambda)>-g^2(\lambda-2\Gm)^{-2}$.
The Rayleigh principle thus asserts that $\Lambda^\epsilon(\lambda)$ is 
the lowest (generalized) eigenvalue of \eqref{E:Sturm} 
if and only if $\Lambda^\epsilon(\lambda) \in (-\infty, \epsilon)$.
Furthermore, such an eigenvalue $\Lambda^\epsilon(\lambda)$ is simple.
Our aim is then to find a $\el$ such that $\Lambda^\epsilon(\el)=-k^2(\pi/L)^2$.
There may be multiple solutions, for instance, corresponding to different values of $k$.
Here, we restrict ourselves to finding one for $k=1$.
Note that $\Lambda^\epsilon(\lambda)$ is a $C^1$ function of~$\lambda$.

First, for $\lambda \in [gL/\pi+2\Gamma_{\inf},\infty)$ it is straightforward that 
\begin{align*}
\int^0_{-\infty} \big(a\A  (\pi/L)^2v^2&+a^3\A(v')^2+\epsilon a^3\A v^2\big)dp\\
&\geq \int^0_{-\infty} \big(a\A  (\pi/L)^2v^2+a^3\A(v')^2\big)dp\\
&\geq 2\pi/L \int^0_{-\infty}  a^2(\lambda)vv'dp
\geq 2g\int^0_{-\infty}vv' dp = gv^2(0)
\end{align*}
for every $v \in H^1((-\infty,0))$. The second inequality uses the Schwarz inequality.
Thus, $R^\epsilon(\lambda)\geq -(\pi/L)^2$, 
and in turn $\Lambda^\epsilon(\lambda) \geq -(\pi/L)^2$.

Next, provided that \eqref{C:bifurcation} holds, one can show that
\begin{align*}
\Lambda^\epsilon(-2\Gamma_{\inf})&\leq R^\epsilon(e^p;-2\Gamma_{\inf}) \\&=
\frac{-g+\int^0_{-\infty} a^3(-2\Gamma_{\inf}) e^{2p}\;dp
+\epsilon \int^0_{-\infty} a^3(-2\Gamma_{\inf})e^{2p}\; dp}
{\int^0_{-\infty} a(-2\Gamma_{\inf}) e^{2p}\;dp}<-(\pi/L)^2.
\end{align*}
By continuity, then, there exists $\el \in (-2\Gamma_{\inf}, gL/\pi+2\Gamma_{\inf}]$ 
such that $\Lambda^\epsilon(\el)=-(\pi/L)^2$.

The uniqueness of $\el$ follows by that 
$\Lambda^\epsilon(\lambda)$ is a monotonically increasing function of $\lambda$
as long as $\Lambda^\epsilon<0$. 
The proof is nearly identical to that in \cite[Lemma 3.4]{CoSt04} 
in the finite-depth case (and when $\epsilon=0$), and hence it is omitted. 

Next, let $\eP \in H^1((-\infty,0))$ be an eigenfunction of \eqref{E:Sturm} 
corresponding to the generalized eigenvalue $\mu(\el)=-(\pi/L)^2$. 
Since the eigenvalue is simple, $\eP$ is unique up to constant multiple.
It follows from regularity theory that $\eP$ is smooth. 
Moreover, since \[(\el+2\Gamma_{\inf})^{1/2} \leq a(\el) \leq (\el+2\Gamma_{\sup})^{1/2},\]
where $\Gamma_{\sup}=\sup_{-\infty<p<0} \Gamma(p)$,  
the comparison theorem for second-order ordinary differential equations 
\cite[Chapter~8]{CoLe55} asserts that $\eP (p)>0$ for all $-\infty<p<0$ and it decays exponentially:
\[
|\eP(p)| \leq K \exp\left(p\frac{(\el+2\Gamma_{\inf})^{1/2}}{(\el+2\Gamma_{\sup})^{3/2}}\right)
\]
for some constant $K>0$. Therefore, $\varphi^\epsilon(q,p)=\eP(p)\cos (\pi/L)q$ is in~$X$.

(ii) The assertion follows since $\Lambda^\epsilon(\lambda)$ is continuous 
and it is nondecreasing function of $\epsilon \geq 0$. 
\end{proof}

It follows as an application of the local bifurcation theorem from a simple eigenvalue \cite{CrRa71}
that for each $0<\epsilon <1$ there emanates from $(\el,0)$ a local curve in $\mathbb{R} \times X$
of solutions to \eqref{E:approx}. 

\begin{proposition}[Local bifurcation for approximate problems]\label{P:local}
Let $L>0$ be held fixed and let 
$\gamma \in C^{1+\alpha}([0,\infty))$, $\alpha \in (0,1)$, satisfy \eqref{C:bifurcation}.

For each $0<\epsilon <1$, there exist $s_0>0$ sufficiently small and a $C^1$-curve 
\[ \mathcal{C}^\epsilon_{loc} = \{ (\lambda(s), w(s)) \in \mathbb{R} \times X : |s| <s_0 \}, \]
where each point $(\lambda(s),w(s))$ in the curve is a solution to \eqref{E:approx}. 

At $s=0$, the solution $(\lambda(0), w(0))=(\el,0)$ 
corresponds to a trivial shear flow under the flat surface.
At $s>0$, the corresponding nontrivial solutions $(\lambda(s), w(s))$ enjoys the following properties:
\begin{alignat}{2}
\lambda(s)&=\lambda^\epsilon+O(s) \qquad & &\text{as $s \to 0$}, \notag \\
w(s)&=s \Phi^\epsilon(p)\cos (\pi/L)q + O(s^2) \qquad & &\text{as $s\to 0$}.\label{E:w(s)}
\end{alignat}
\end{proposition}

The proof is almost identical to that in \cite[Appendix~A]{Hur06}, and hence it is omitted.

One may replace the sufficient condition \eqref{C:bifurcation} for bifurcation 
by a more general condition that the system
\begin{equation}\label{C:bifurcation-g}
\begin{cases}
(a^3\A v')' -\epsilon a^3\A v=-\left(\pi/L\right)^2 a(\lambda) v
\qquad \text{for } p \in (-\infty, 0){,} \\
\lambda^{3/2}v'(0)=gv(0) \end{cases}
\end{equation}
with $v,v' \to 0$ as $p \to -\infty$, admits a nontrivial solution 
for some $\lambda \in (-2\Gamma_{\inf}, \infty)$.

\begin{remark*}[Bifurcation at $\epsilon=0$]\rm
A candidate bifurcation point $\lambda^0$ and a transversal solution $\varphi^0$ of 
$F_w(\lambda^0,0)[\varphi]=0$, when $\epsilon=0$ in Lemma \ref{L:eigen} 
exist for the singular problem \eqref{E:main}. 
But, even the local bifurcation for \eqref{E:main} is {\em singular}.
For, $F_w(\lambda,0): X \to Y$ is not Fredholm, and thus 
local bifurcation theorem (due to Crandall and Rabinowitz \cite{CrRa71}, for instance) is not applicable.
Power series methods may yield a direct proof of small-amplitude solutions of \eqref{E:main},
but since the present purpose is global existence theory, we proceed with singular theory of bifurcation.

For the future reference, recorded here are that 
$\lambda^0 \in (-2\Gamma_{\inf}, gL/\pi-2\Gamma_{\inf}]$ is the unique solution of
\[ \Lambda^0(\lambda)=\inf_{v \in H^1((-\infty,0))} 
\frac{-gv^2(0) + \int^0_{-\infty} a^3\A(v')^2dp\,}{\int^0_{-\infty} a\A v^2dp}=-\left(\pi/L\right)^2\]
and that $\varphi^0(q,p)=\Phi^0(p)\cos (\pi/L)q \in X$ is 
a unique (up to constant multiple) nontrivial solution of $F_w(\lambda^0,0)[\varphi]=0$. 
Moreover, $\Phi^0(p)>0$ for all $-\infty<p<0$ and it is smooth and decays exponentially as $p \to -\infty$.
Finally, $\Lambda^0(\lambda)$ is monotonically increasing with $\lambda$ as long as $\Lambda^0<0$.
\end{remark*}

%%%%%%%%%%%%%%%%%%%%%%%%%%%%%%%%%%%%%%%%%%%%%%%%%%%%%%%%%%%%%%%%%%%%%%%%%%%%%%%%%%%%%%%%%%%%%%%%%%%%%%
\subsection{Global bifurcation for approximate problems}\label{SS:global-approx}

Now the presentation is for $0<\epsilon <1$ 
the existence in the large of solutions to \eqref{E:approx} via global bifurcation theory.

Throughout the subsection, $\delta>0$ and $0<\epsilon<1$ are held fixed. 
Let $\Sde$ be the closure in $\mathbb{R}\times X$
of the set of all nontrivial solution pairs $(\lambda, w) \in\OO$ to \eqref{E:approx}, 
where $F^\epsilon$ is given in \eqref{D:F^epsilon} and $\OO$ is defined in \eqref{D:O_delta}.
Let $\C \subset \mathbb{R} \times X$ be the connected component of $\Sde$
containing the bifurcation point $(\el,0)$ determined in Lemma~\ref{L:eigen}.
The local curve of solutions $\mathcal{C}^\epsilon_{loc}$ constructed in Proposition~\ref{P:local} 
is contained in $\C$. %if $s_0>0$ of the curve is sufficiently small.

Recall from Section \ref{S:degree} that for each $\delta>0$ and for each $\epsilon>0$, 
%the operator $F^\epsilon: \OO \to Y$ is proper and that for each $(\lambda,w) \in \OO$ 
%the operator $F^\epsilon_w(\lambda,w)$ is Fredholm of index zero
%and it satisfies spectral properties in Lemma \ref{L:spectral}. As such, 
the generalized {\em degree} due to Healey and Simpson \cite{HeSi98} 
is successfully defined to $F^\epsilon$ in the set $\OO$. 
The following global bifurcation result is then immediate.

\begin{proposition}\label{T:global}
For every $\delta>0$ and for every $0<\epsilon< 1$, at least one of the following holds:
\renewcommand{\labelenumi}{{\rm(\roman{enumi})}}
\begin{enumerate}
\item $\C$ is unbounded in\/ $\mathbb{R} \times X$;
\item $\C$ contains another trivial solution pair\/ $(\lambda,0)$ with $\lambda \neq \el$;
\item $\C$ meets $\partial \OO$.
\end{enumerate}
\end{proposition}

The proof is almost identical to that of \cite[Theorem~1.3]{Rab71}
except that here the generalized degree defined in \cite{HeSi98} is used 
in place of the Leray-Schauder degree, and hence it is omitted.

The remainder of this subsection is devoted to refining the result of Proposition~\ref{T:global}
and to establishing some properties of $\C$ 
useful in characterizing $\C$ in terms of bounded subsets of $\OO$.

Demonstrated below is how exploitation of symmetry rules out 
the second alternative from Proposition~\ref{T:global}. 
In order to state the result precisely, let us denote the nodal set and its boundaries by
\begin{align*}
R^-&=\{(q,p) \in R: -L<q<0\}, & T^-&=\{ (q,0):  -L <q<0\}{,} \\
\partial R^-_l&=\{(-L,0): -\infty<p<0\},& \partial R^-_r&=\{(0,p):-\infty<p<0\}.
\end{align*}
First, upon examination of \eqref{E:w(s)}, the following nodal properties 
\begin{gather}
w_q>0 \quad \text{ in }R^- \cup T^-{,} \label{nodal1}\\
w_{qq}>0 \quad \text{on }\partial R^-_l,\qquad
w_{qq}<0 \quad \text{on }\partial R^-_r{,} \label{nodal2}\\
w_{qq}(-L,0)<0, \qquad w_{qq}(0,0)>0\label{nodal3}.
\end{gather}
hold along the local bifurcation curve $\mathcal{C}^\epsilon_\delta$.
The proof is based on properties of $\Phi^\epsilon \cos(\pi/L)q$ in Proposition \ref{P:local},
and it is detailed in \cite[Lemma C.1]{Hur06}.

Next, 
since $A^\epsilon(\lambda, w)$ is a uniformly elliptic second-order linear partial differential operator 
and since $B(\lambda,w)$ is uniformly oblique boundary operator, 
the maximum principle, the Hopf boundary lemma and the edge-point lemma applies 
to the boundary value problem 
\[(A^\epsilon(\lambda,w), B(\lambda,w))[w_q]=0\] in $R^-$.
Consequently, each nontrivial solution $w$ of \eqref{E:approx} with $(\lambda,w)\in\C \setminus (\el,0)$ 
possesses \eqref{nodal1}, \eqref{nodal2} and 
\begin{equation}\label{false-nodal} 
\text{either $w_{qq}(0,0)<0$ or $w_{qqp}(0,0)>0$, \, \,
either $w_{qq}(-L,0)>0$ or $w_{qqp}(-L,0)<0$},
\end{equation}
unless $\C$ contains a trivial solution $(\lambda,0)$ other than $(\lambda^\epsilon,0)$. 
The proof is in \cite[Lemma C.3]{Hur06}. 
On the other hand, by the variational characterization of $\lambda^\epsilon$, it follows that 
if $(\lambda,0) \in \C$ then $\lambda=\lambda^\epsilon$. 
See \cite[Lemma 4.8]{Hur06}. See also Lemma \ref{L:nodal}.
Therefore, the nodal properties \eqref{nodal1}, \eqref{nodal2} and \eqref{false-nodal}
preserve along the continuum $\C \setminus (\el, 0)$. 

Finally, our task is to prove \eqref{nodal3} assuming that \eqref{false-nodal} holds.
By evenness, $w_q(0,p)=w_{qqq}(0,p)=0$ for $-\infty<p\leq 0$ and $w_{qp}(0,p)=0$ for $-\infty<p\leq 0$. 
Moreover, from \eqref{nodal2} it follows $w_{qq}(0,0)\leq 0$. 
Differentiating the boundary condition $F_2(\lambda,w)=0$ twice in the $q$-variable results in 
\begin{multline*}
gw_{qq}(\lambda^{-1/2}+w_p)^2+4gw_q(\lambda^{-1/2}+w_p)w_{qq}+(2gw-\lambda)w_{pq}^2\\
+(2gw-\lambda)(\lambda^{-1/2}+w_p)w_{pqq}+w_{qq}^2+w_qw_{qqq}=0
\end{multline*}
on $T$. Evaluated at $(0,0)$, it reduces to 
\[gw_{qq}(\lambda^{-1/2}+w_p)^2+(2gw-\lambda)(\lambda^{-1/2}+w_p)w_{pqq}+w_{qq}^2=0.
\]
Suppose that $w_{qq}(0,0)=0$ so that \eqref{false-nodal} would dictate $w_{pqq}(0,0)>0$. 
Then, the above equation would reduce to 
\[(2gw-\lambda)(\lambda^{-1/2}+w_p)w_{pqq}=0 \qquad \text{at $(0,0)$.}
\] 
Since $2gw-\lambda$ and $\lambda^{-1/2}+w_p$ are nonzero, however, 
$w_{pqq}(0,0)=0$ would hold. This proves the assertion by contradiction. 

Let us define the open set\footnote{Note that
\eqref{nodal1}-\eqref{nodal3} define an open set in $X$, while 
\eqref{nodal1}, \eqref{nodal2} and \eqref{false-nodal} will not define an open set.
See \cite{CoSt07} for a more detailed discussion.}
\begin{equation}\label{D:N}
\mathcal{N}=\{\,w \in X: w \text{ satisfies \eqref{nodal1}-\eqref{nodal3}}  \}
\end{equation}
and summarize the results.

\begin{theorem}[Global bifurcation for approximate problems]\label{T:global1a}
For each $\delta>0$ and for each $0<\epsilon< 1$, the global continuum $\C$ 
either is unbounded in\/ $\mathbb{R} \times X$ or intersects $\partial \OO$. 
Each nontrivial solution lying on $\C$ has the nodal configuration \eqref{nodal1}-\eqref{nodal3},
i.e.,~$\C \setminus (\el,0) \subset \mathbb{R} \times \mathcal{N}$.
\end{theorem}

The following result describes global bifurcation in terms of bounded open sets in $\OO$,
and it is crucial in the next subsection in obtaining 
a global continnum of nontrivial solutions to \eqref{E:main}.  

\begin{theorem}[\mbox{see \cite[Theorem A6]{AmTo81}}]\label{T:A6}
For each $\delta>0$, let $S\subset \OO$ be a closed set with $(\lambda,0) \in S$ and let  
every bounded subset of $S$ be relatively compact in $\mathbb{R} \times X$.
Let $C$ be the maximal connected subset of $S$ containing $(\lambda,0)$.
Then $C$ either is unbounded in $\mathbb{R} \times X$ or meets $\partial \OO$
if and only if $\partial U \cap S \neq \emptyset$
for every bounded open set $U$ in $\OO$ with $(\lambda,0) \in U$.
\end{theorem}

Based on Wyburn's lemma \cite{Why} in topology,
the proof is identical to that of \cite[Theorem~A6]{AmTo81}, 
except that $S$ is confined in $\OO \subset \mathbb{R} \times X$ 
so that $C$ can be bounded by intersecting $\partial \OO$ 
if it is not unbounded in $\mathbb{R}\times X$. %The proof hence is not detailed here. 

\begin{corollary}\label{C:AT}
For each $\delta>0$ and for $0<\epsilon<1$ sufficiently small,
if $U$ is a bounded open set in $\OO$ with $(\lambda^0,0) \in U$ and $\overline{U} \subset \OO$
then $\partial U \cap \C \neq \emptyset$.
\end{corollary}

The assertion follows once Theorem~\ref{T:A6} applies to the solution branch $\C$ of \eqref{E:approx}.
Indeed, the result of Lemma~\ref{L:proper} says that  
any bounded subset of $\Sde$ is relatively compact in $\mathbb{R} \times X$, 
and the result of Lemma \ref{L:eigen}(b) says that 
$(\el, 0) \in U$ for sufficiently small $\epsilon>0$.

%%%%%%%%%%%%%%%%%%%%%%%%%%%%%%%%%%%%%%%%%%%%%%%%%%%%%%%%%%%%%%%%%%%%%%%%%%%%%%%%%%%%%%%%%%%%%%%%%%%%%%
\subsection{Global existence of rotational Stokes waves}\label{SS:global}
A global connected set in $\mathbb{R} \times X$ 
of nontrivial solutions to the singular problem \eqref{E:main} is constructed.

For each $\delta>0$ let 
\begin{equation}\label{D:S_delta} 
\begin{split}
\mathcal{S}_\delta=\{\, (\lambda, w)\in \OO : & F(\lambda, w)=0,\;\;w \not\equiv 0, \;\; w\in \mathcal{N}  \\ 
&w_q \in O(|p|^{-1-\rho})\text{ as $p \to -\infty$} \,\}\cup \{(\lambda^0, 0)\}
\end{split}\end{equation}
for some $\rho>0$. In other words, $\mathcal{S}_\delta$ consists of nontrivial solutions to \eqref{E:main} 
with the nodal properties \eqref{nodal1}-\eqref{nodal3} 
and with the decay condition that $w_q \in O(|p|^{-1-\rho})$ as $p \to -\infty$,
plus the bifurcation point $(\lambda^0,0)$. 
Let $\mathcal{C}'_\delta$ be the maximal connected component of 
the closure in $\mathbb{R} \times X$ of $\mathcal{S}_\delta$ containing $(\lambda^0,0)$. 
Our goal is to show that $\mathcal{C}'_\delta$ is a continnum 
of nontrivial solutions of \eqref{E:main} with the desired properties.

In order to apply Theorem \ref{T:A6} to $\mathcal{S}_\delta$ and $\mathcal{C}'_\delta$, 
we need to establish that $\mathcal{S}_\delta$ is closed and that 
every bounded subset of $\mathcal{S}_\delta$ is relatively compact in $\mathbb{R} \times X$. 
The relative compactness of $\mathcal{S}_\delta$ requires,
as we shall see in the proof of Lemma \ref{L:compact}, 
a certain uniform control at the infinite bottom of functions in a bounded subset of~$\mathcal{S}_\delta$.
The following result in the spirit of the Phragm\'en-Lindel\"of theorem furnishes 
a uniform decay as $p \to -\infty$ of solutions $w$ of \eqref{E:main} bounded under the norm of~$X$.

\begin{lemma}[Exponential decay]\label{L:exp-decay}
Let $\gamma \in C^{1+\alpha}([0,\infty))$, $\alpha \in (0,1)$ satisfy 
$\gamma(r) \in O(r^{-2-2\rho})$ as $r \to \infty$. 
For each $\delta>0$, if $(\lambda,w) \in \mathcal{S}_\delta$ and if $|\lambda|+\|w\|_{X} <M$
for some $M>0$, then $w_q$ enjoys an exponential decay property
\begin{equation}\label{E:exp-decay}
|w_q(q,p)| \leq M (2-e^{\beta q}) e^{\sigma p} \qquad \text{for all}\quad (q,p) \in R,
\end{equation}
where $\beta>\frac{KM^2}{2\delta^2}$ for some constant $K$ 
and the constant $\sigma>0$ sufficiently small is given in \eqref{E:sigma}.
\end{lemma}

\begin{proof}
Differentiating the equation $F_1(\lambda,w)=0$ in the $q$-variable yields that
\begin{align*}
0=&(1+w_q^2)w_{qpp}-2(a^{-1}(\lambda)+w_p)w_q w_{qpq}
+(a^{-1}(\lambda)+w_p)^2 w_{qqq} \\ 
&+\left(-2w_qw_{pq}+3\gamma\Pf(a^{-1}\A+w_p)^2\right) w_{qp}
+\left(2w_{q}w_{pp}-2\gamma\Pf a^{-3}\A w_q  \right)w_{qq} \\
=:&A^P(\lambda, w)[w_q]+b_1(\nabla w, \nabla^2 w)w_{qp} +b_2(\nabla w, \nabla^2 w)w_{qq}
\end{align*}
in $R$, where $A^P(\lambda, w)$ is the principal part of $A(\lambda,w)$, and
$b_1(\nabla w, \nabla^2 w)$ and $b_2(\nabla 2, \nabla^2 w)$ consist of 
quadratic polynomial expressions in $\nabla w$ and $\nabla^2 w$. 
Since $\gamma(-p) \in O(|-p|^{-2-2\rho})$ and $w_q \in O(|p|^{-1-\rho})$ as $p \to -\infty$, 
it follows that 
\[ b_1(\nabla w, \nabla^2 w), b_2(\nabla 2, \nabla^2 w) \in O(|p|^{-1-\rho})
\qquad \text{as $p \to -\infty$.}\] 
Since $(\lambda, w) \in \mathcal{O}_\delta$, the operator $A^P(\lambda, w)$ is uniformly elliptic. 
Moreover, $a^{-1}(\lambda)+w_p>\delta>0$ in $\overline{R}$. It is straightforward that 
\[|b_1(\nabla w, \nabla^2 w)|, |b_2(\nabla w, \nabla^2w)| \leq KM^2 \qquad \text{ for some $K>0$.}\]
%Alternatively, $b_2(\nabla w, \nabla^2 w)w_{qq}=2(w_{pp}-\gamma(-p)a^{-3}(\lambda))w_qw_{qq}$,
%and \[ |w_{pp}-\gamma(-p)a^{-3}(\lambda)| \leq KM.\] 

Let $v=w_q$, and we consider the function in the nodal set $R^-=\{ (q,p) \in R: -L<q<0 \}$. 
It is immediate that $v$ solves in $R^-$ 
the following elliptic second-order partial differential equation 
\[L[v]:=A^P(\lambda, w)[v]+b_1(\nabla w, \nabla^2w)v_p+b_2(p,w_{pp})vv_q =0.\] 
By the nodal property that $w \in \mathcal{N}$, it follows that $v>0$ in $R^-$. 
By evenness of $w$, furthermore, $v(q,p)=0$ for $q=0$ or $q=-L$ for all $-\infty <p \leq0$. 

Let us define the auxiliary function 
%\begin{equation}\label{E:z} 
%z(q,p)=\frac{1}{\mu}(e^{\mu v(q,p)}-1)\end{equation}
%in $R^-$, where the constant $\mu>0$ will be determined in the course of the proof. 
%It is straightforward that 
%\begin{align*}
%A^P(\lambda,w)[z]=\frac{e^{\mu v}}{\mu} \Big( &\mu A^P(\lambda,w)[v] \\
%&+\mu^2 \left((1+w_q^2)v_p^2-2(a^{-1}(\lambda)+w_p)w_qv_pv_q
%+(a^{-1}(\lambda)+w_p)^2 v_q^2\right)\Big) \\
%\geq e^{\mu v} \Big(& -\left( b_1(\nabla w, \nabla^2 w)v_p+b_2(p, w_{pp}) vv_q\right)
%+4\mu \delta^2 (v_p^2+v_q^2) \Big).
%\end{align*}
%The inequality uses the uniform ellipticity of $A^P(\lambda, w)$. 
%Since $|v_q| \leq M$, for $\mu$ small so that $\mu<1/M$, accordingly, 
%\begin{align*} 
%L[z]\geq & \mu e^{\mu v} b_2(p,w_{pp})v_q(-\mu v+e^{\mu v}-1) 
%+ 4\mu \delta^2 e^{\mu v}(v_p^2+v_q^2) \\
%> & -\mu e^{\mu v} eKM^2.
%\end{align*}
%Let us define another auxiliary function
\begin{equation}\label{E:exp-function}
u(q,p)=A(2-e^{\beta q}) e^{\sigma p}- v(q,p),
\end{equation}
in $R^-$, where positive constants $A$, and $\beta$, $\sigma$ 
will be determined in the course of the proof. 
%Our aim is to show that $u(q,p) \geq 0$ in $R^-$. 

%Let $-\infty<p_0<0$ be held fixed, and let us for now restrict the analysis to the rectangle 
%$R^-_{p_0}=\{ (q,p) :-L<q<0, \, p_0<p<0\}$. Recalling that $L[z]>0$, 
It is straightforward that 
\begin{equation}\label{E:L[u]}
\begin{split} 
L[u] \leq Ae^{\sigma p} \Big(\sigma^2  
(1+w_q^2)(2-& e^{\beta q}) -2\sigma \sigma (a^{-1}(\lambda)+w_p)w_qe^{\beta q}  
-\beta^2 (a^{-1}(\lambda)+w_p)^2 e^{\beta q} \\
&+\sigma b_1(\nabla w, \nabla^2 w)(2-e^{\beta q})+ \beta b_2(\nabla w, \nabla^2w) e^{\beta q}\Big) 
-L[v] \\
\leq Ae^{\sigma p}\Big( 2\sigma^2(1+M^2)+& 2\beta \sigma KM^2+2\sigma KM^2
-\beta^2\delta^2e^{\beta q}+\beta KM^2e^{\beta q}\Big).
\end{split}
\end{equation}

Let $\beta=\max\left(1, \frac{KM^2}{2\delta^2}\right)$ so that 
\[-\beta^2\delta^2e^{\beta q}+ \beta KM^2e^{\beta q}=\beta e^{\beta q}(-\beta \delta^2+KM^2)
\leq -\frac12 e^{-\beta L} KM^2.\]
Subsequently, let us choose $\sigma>0$ small enough that
\begin{equation}\label{E:sigma}
2\sigma^2(1+M^2) -4\beta \sigma KM^2-\frac12 e^{-\beta L}KM^2<0. 
\end{equation}
Then, $L[u]< 0$ in $R^-$. 

Out task now is to examine $u$ on the boundaries of $R^-$. 
At the top boundary $\{(q,0): -L<q<0\}$, we have
\[ u(q,p)=A(2-e^{-\beta q})-v(q,p) \geq A -M= 0\]
if $A=M$. 
Since $v(q,p)=0$ on the side boundaries $\{ (q,p): q=0 \text{ or }q=-L, \, p_0 <p<0\}$, it follows that
\[ u(q,p) = M(2-e^{-\beta q}) e^{\sigma p} > 0 \qquad \text{for}\quad q=0\text{ or }q=-L, \, -\infty <p<0.\]
%Moreover, since $v(q,p) \to 0$ as $p \to -\infty$, it follows 
%\[ u(q,p)=a((L+1)^2-q^2)e^{\sigma p}-bp-v(q,p) \to 0\quad \text{as }p \to -\infty.\]
%In summary, $u(q,p) \geq 0$ on the boundaries of $R^-$. 
In summary, $L[u]<0$ in the domain $R^-$ and $u \geq 0$ on the boundaries of $R^-$.
Since $A^P(\lambda,w)$ is uniformly elliptic, and 
$b_1(\nabla w, \nabla^2 w), b_2(\nabla w, \nabla^2 w) \in O(|p|^{-1-\rho})$ 
as $p\to -\infty$, by the Phragm\'en-Lindel\"of theorem \cite{Gil}, it follows that $u\geq 0$ in $R^-$. 
That is, 
\[0\leq v(q,p) \leq M(2-e^{\beta q}) e^{\sigma p} \quad \text{in }R^-.\]
Repeating the above argument for $-v$ on $R^+$ yields an analogous inequality. 
This completes the proof.
\end{proof}

The above proof applies to $F_1^\epsilon(\lambda, w)=F_1(\lambda,w)-\epsilon w=0$, $\epsilon>0$,
mutatis mutandis to yields an exponential decay of solutions to \eqref{E:approx},
analogous to \eqref{E:exp-decay}. 

\begin{corollary}\label{C:exp-decay}
For each $\delta>0$ and for each $0<\epsilon <1$ 
if $(\lambda,w) \in \mathcal{C}^\epsilon_\delta$ and if $|\lambda|+\|w\|_{X} <M$
for some $M>0$, then $w_q$ enjoys an exponential decay property  
\begin{equation}\label{E:exp-decay'}
|w_q(q,p)| \leq M (2-e^{\beta q}) e^{\sigma p} \qquad \text{for all}\quad (q,p) \in R 
\end{equation}
where $\beta$ and $\sigma$ are as in Lemma \ref{L:exp-decay}.
\end{corollary}

The exponential decay of solutions of \eqref{E:main} in \eqref{E:exp-decay} 
establishes the relative compactness of $\mathcal{S}_\delta$ in $\mathbb{R} \times X$. 

\begin{lemma}[Relative compactness]\label{L:compact}
Suppose that $\gamma \in C^{1+\alpha}[(0,\infty))$, $\alpha \in (0,1)$,
and $\gamma(r) \in O(r^{-2-2\rho})$ as $r  \to \infty$ for some $\rho>0$. 
For each $\delta>0$, any bounded subset of $\Ss$ is relatively compact in\/ $\mathbb{R} \times X$.
\end{lemma}

The nonlinear operator $F$ in $\OO \subset \mathbb{R} \times X$ is not (locally) proper,
as is discussed in Section \ref{SS:methods}. Indeed, its limiting problem 
$$v_{pp}+(\lambda+2\Gamma_\infty)^{-1}v_{qq}=0$$
in the infinite strip $\{(q,p): -L<q<L,\, -\infty<p<\infty\}$ admits infinitely many solutions in the $C^0$ class. 
Here, compactness is established only for the solution set of $F$, not for the operator itself.

\begin{proof}
Let $\{(\lambda_j,w_j)\} \subset \Ss$ be a sequence in $\mathbb{R} \times X$ with 
$|\lambda_j|+\|w_j\|_X < M$ for all $j$ for some $M>0$. Note that, $F(\lambda_j, w_j)=0$ for all $j$.
It is immediate that $\{\lambda_j\}$ has a convergent subsequence in $\mathbb{R}$.
By possibly relabeling the index, let $\lambda_j \to \lambda$ as $j \to \infty$.
Moreover, it is immediate that $w_j \to w$ as $j \to \infty$ for some $w$ 
in $\overline{R'}$ for any bounded subset $R'$ of $R$. By continuity, $F(\lambda, w)=0$. 
Our aim is to show that $\{w_j\}$ has a subsequence converging to $w$ as $j \to \infty$ in~$X$.

As is indicated in the proofs of Lemma \ref{L:fred} and Lemma~\ref{L:proper},
a crucial step in showing the convergence of $\{w_j\}$ in $X$ is 
to obtain the convergence of $\{w_j\}$ in the $C^0_{per}(\overline{R})$ norm.
Since $\{w_j\}$ converges in $C^0_{per}(\overline{R'})$ for any $R' \subset R$ bounded,
it entails to show that $\{w_j\}$ decays as $p \to -\infty$ uniformly for $j$. 

It is convenient to write  
\begin{equation}\label{E:w-decomp}
w_j(q,p) = \int^q_0 \partial_q w_j(q',p) dq' + w_j(0,p).
\end{equation}
Note that $\{\partial_q w_j \}$ and  $\{w_j(0,\cdot )\}$ are bounded in 
$C^{2+\alpha}_{per}(\overline{R})$ and $C^{3+\alpha}_{per} ((-\infty,0])$, respectively.

Since $\gamma(-p)\in O(|p|^{-2-2\rho})$ and $\partial_q w_j \in O(|p|^{-1-\rho})$ as $p \to -\infty$ 
for some $\rho$, and  since $|\lambda_j|+ \|w_j\|_X \leq M$ for all $j$, 
the result of Lemma \ref{L:exp-decay} states that $\partial_q w_j$ decays exponentially as $p \to -\infty$ 
uniformly for $j$. More precisely,  
\[ | \partial_q w_j(q,p)| \leq Ce^{\sigma p} \qquad \text{ for all $(q,p) \in R$}\]
where $C>0$ and $\sigma>0$ depends only on $\delta$ and $M$. 
An argument of Ascoli type then applies to assert that 
the first term in \eqref{E:w-decomp} has a subsequence converging in $C^0_{per}(\overline{R})$. 
Moreover, by the dominated convergence theorem, 
\[ \int^q_0 \partial_q w_j(q'p) dq' \to \int^q_0 \partial_q w (q',p) dq' \qquad \text{as $j \to \infty$.}\]

Next is to examine the latter term in \eqref{E:w-decomp}. 
Since $\partial_qw_j$ decays exponentially like \eqref{E:exp-decay} as $p \to -\infty$ uniformly for $j$, 
it follows from the classical gradient estimate for elliptic equations \cite[p.37]{GiTr01} that 
\begin{equation}\label{E:w-decay2}
|\partial^2_q w_j(0,p) | \leq C' e^{\sigma p} \qquad \text{for all $p \in (-\infty,0)$}
\end{equation}
where $C'>0$ is independent of the index $j$, and $\sigma$ is the same as in \eqref{E:exp-decay}.
When restricted on the half-line $q=0$, by evenness of $w_j$, 
the equation $F_1(\lambda_j,w_j)=0$ reduces to
\begin{equation*}\label{E:q=0}
\partial^2_p w_j +(a^{-1}\A +\partial_p w_j)\partial^2_q w_j 
+\gamma\Pf(a^{-1}\A +\partial_p w_j)^3-\gamma\Pf a^{-3}\A =0.
\end{equation*}
Since $\gamma \in O(r^{-2-2\rho})$ as $r \to \infty$ for some $\rho>0$,
and $\partial_q^2 w_j(0,p)$ decays exponentially like \eqref{E:w-decay2} as $p \to -\infty$ 
uniformly for $j$, then it follows that $w_j(0,p)$ decays as $p \to -\infty$ uniformly for~$j$.  
Again, an argument of Ascoli type asserts that $\{w_j(0,p)\}$
has a subsequence which converge in $C^0((-\infty,0])$. 
In turn, $\{w_j\}$, possibly after relabeling, converges to $w$ in $C^0(\overline{R})$.

The remainder of the proof is nearly identical to that of Lemma~\ref{L:proper}, 
and we only outline the various stages of the proof.

As is done in \eqref{E:w-decomp1} in the proof of Lemma~\ref{L:proper}, 
we decompose $F$ as
\begin{align*}
F_1(\lambda,w) &= A^P(\lambda,w)[w]+f_1(\lambda,w){,} \\
F_2(\lambda,w) &= B^P(\lambda,w)[w]+f_2(\lambda,w),
\end{align*}
where $A^P(\lambda,w)$, $B^P(\lambda,w)$, $f_1(\lambda,w)$ and $f_2(\lambda, w)$ 
are the same as in the proof of Lemma~\ref{L:proper}. We may write 
\begin{align*}
A^P(\lambda_j,w_j)[w_j-w]=-(A^P(\lambda_j,w_j)&-A^P(\lambda,w))[w]
-(f_1(\lambda_j,w_j)-f_1(\lambda,w)),\\
B^P(\lambda_j,w_j)[w_j-w]=-(B^P(\lambda_j,w_j)&-B^P(\lambda,w))[w]
-(f_2(\lambda_j,w_j)-f_2(\lambda,w)).
\end{align*}

Since $A^P(\lambda,w)$ is uniformly elliptic and $B^P(\lambda, w)$ is uniformly oblique, 
an estimate of Schauder type~\cite{ADN59} states that 
\begin{equation}\label{E:Schauder10}
\| w_j-w \|_X \leq C(\|A^P(\lambda_j,w_j)[w_j-w]\|_{Y_1} +
\|B^P(\lambda_j,w_j)[w_j-w]\|_{Y_2} +\| w_j-w\|_Z)
\end{equation}
holds, where $C>0$ is independent of $j$.
Since $\{ w_j\}$ is bounded in $X$ and it converges in $C^0_{per}(\overline{R})$ 
an interpolation inequality (see \eqref{E:interp} in Lemma~\ref{L:proper} or \cite[Theorem 3.2.1]{Kry96}) 
asserts that $w_j \to w$ in $C^3_{per}(\overline{R})$. Accordingly,
\[
\| A^P(\lambda_j,w_j)[w_j-w]\|_{Y_1} \to 0 \qquad \text{as $j \to \infty$}.
\]
Moreover, by the standard Schauder theory and the embedding properties in the bounded domain,
\[ \|B^P(\lambda_j,w_j)[w_j-w]\|_{Y_2} \to 0 \qquad \text{as $j \to \infty$}. \]
The Schauder estimate \eqref{E:Schauder10} then dictates that 
$w_j \to w$ in $C^{3+\alpha}(\overline{R})$. 
Furthermore, by continuity, $w \in X$ and $F(\lambda,w)=0$. 

Finally, since $\{\partial_q w_j\}$ decays exponentially as $p \to -\infty$, 
it follows that $w_q \in O(|p|^{-1-\rho})$ as $p \to -\infty$. 
By Lemma \ref{L:exp-decay}, in fact, $w_q$ decays exponentially as $p \to -\infty$.
This completes the proof.
\end{proof}

Next is to show that $\mathcal{S}_\delta$ is closed. If $\{(\lambda_j, w)\}$ in $\mathcal{S}_\delta$
converges to $(\lambda,w)$ as $j \to \infty$ and if $w$ is not identically zero, 
then, by continuity, $(\lambda,w)$ is a nontrivial solution of \eqref{E:main}. 
Moreover, by Lemma \ref{L:exp-decay} and Lemma \ref{L:compact},
$w_q$ decays exponentially as $p \to -\infty$. That means, $(\lambda, w) \in \mathcal{S}_\delta$.
Thus, it remains to show that if $\{(\lambda_j,w_j)\}$ converges to $(\lambda, 0)$ as $j \to \infty$
then $\lambda=\lambda^0$.
For $0<\epsilon<1$, the analogous property ensures that 
the nodal properties preserve along the global continnum 
$\Cd \setminus (\lambda^\epsilon,0)$ of solutions to \eqref{E:approx}.

\begin{lemma}[Closedness]\label{L:nodal}
For each $\delta>0$, if $(\lambda,0) \in \mathcal{S}_\delta$ then $\lambda=\lambda^0$. 
\end{lemma}

\begin{proof}
The proof is similar to that of \cite[Lemma 4.8]{Hur06}. 
Let $(\lambda,0) \in \mathcal{S}_\delta$ and let 
$\{(\lambda_j,w_j)\} \subset \mathcal{S}_\delta$, $w_j \not\equiv 0$ for each $j$ such that 
$(\lambda_j,w_j) \to (\lambda,0)$ as $j \to \infty$ in $\mathbb{R} \times X$.

Let $v_j=\partial_q w_j/\|\partial_q w_j\|_{C^{2+\alpha}(\overline{R})}$. 
Since every $\partial_q w_j$ is not identically zero in $\overline{R}$, the function $v_j$ is well-defined.
It is straightforward that each $v_j$ satisfies
\[
(A(\lambda_j,w_j),B(\lambda_j,w_j))[v_j]=0.
\]
Since $\{v_j\}$ is bounded in $C^{2+\alpha}(\overline{R})$,
it follows that $v_j \to v$ in $C^2_{per}(\overline{R'})$ for some $v$
for any bounded subset $R'$ of $R$. 
Moreover, by continuity, \begin{equation}\label{E:v}
(A(\lambda,0), B(\lambda,0))[v]=0.\end{equation}
We claim that $v_j \to v$ in $C^{2+\alpha}(\overline{R})$. 
The proof is similar to that of Lemma \ref{L:proper} or Lemma \ref{L:compact},
and thus we only sketch the outline of the proof. 

First, the result of Lemma \ref{L:exp-decay} states that $v_j(q,p)$ decays 
exponentially as $p \to -\infty$ uniformly for $j$. 
Then, by repeating the argument as in Lemma \ref{L:compact}, 
one accomplishes that $v_j \to v$ as $j \to \infty$ in $C^0_{per}(\overline{R})$. 
Since$\|v_j\|_{C^{2+\alpha}(\overline{R})}=1$ for each~$j$, subsequently, 
by interpolation inequality as in \eqref{E:interp}, 
it follows that $v_j \to v$ as $j \to \infty$ in $C^2_{per}(\overline{R})$. 
Next, it is straightforward that $v_j-v$ satisfies
\begin{alignat*}{2}
A(\lambda_j,w_j)[v_j-v]&=
\bigl(A(\lambda,0)&-A(\lambda_j,w_j)\bigr)[v] \quad &\text{in }\, R{,} \\
B(\lambda_j,w_j)[v_j-v]&=\bigl(B(\lambda,0)&-B(\lambda_j,w_j)\bigr)[v]\quad &\text{on }\, T.
\end{alignat*}
Since $A(\lambda_j, w_j)$ is uniformly elliptic and $B(\lambda_j, w_j)$ is uniformly oblique for each $j$, 
the Schauder estimates~\cite{ADN59} yield
\begin{equation}\label{E:Schauder5}
\begin{split}
\|v_j-v\|_{C^{2+\alpha}(\overline{R})} \leq\; C
\bigl(& \|A(\lambda_j,w_j)[v_j-v]\|_{C^{\alpha}(\overline{R})} \\
&+\|B(\lambda_j,w_j)[v_j-v]\|_{C^{1+\alpha}(T)} +\|v_j-v\|_{C^0(\overline{R})} \bigr){,}
\end{split}
\end{equation}
where $C>0$ is independent of index $j$.
Since $\|v_j\|_{C^{2+\alpha}(\overline{R})}=1$ for each $j$ and 
since $A(\lambda_j,w_j)$ is equicontinuous, it follows that 
\[
\|\bigl(A(\lambda,0)-A(\lambda_j,w_j)\bigr)[v]\|_{C^\alpha(\overline{R})} \to 0 \qquad \text{as $j \to \infty$.}
\]
Moreover, by the standard elliptic theory in a bounded domain, it follows that 
\[ \|B^P(\lambda_j,w_j)[w_j-w]\|_{Y_2} \to 0 \qquad \text{as $j\to \infty$.}\]
Therefore, \eqref{E:Schauder5} proves the claim.
Furthermore, $v, \nabla v, \nabla^2 v \in o(1)$ as $p \to -\infty$ 
uniformly for $q$ and $\|v\|_{C^{2+\alpha}(\overline{R})}=1$.

By the periodicity and symmetry consideration, it follows that $v=\partial_q \varphi$ for some $\varphi$. 
Indeed, since each $v_j$ is $2L$-periodic in the $q$-variable and 
since it is of mean zero over one period, 
that is, $$\int^L_{-L} v_j(q,p) dq=0 \qquad \text{for all $p \in (-\infty,0)$,}$$ by continuity, 
$v$ is also $2L$-periodic in the $q$-variable and it is of mean zero over one period. 
Thus, the assertion follows, where $\varphi$ is $2L$-periodic in the $q$-variable. 

Let us write \eqref{E:v} as 
\begin{equation}\label{E:lambda0}
(A(\lambda,0), B(\lambda, 0))[\partial_q \varphi]=0.
\end{equation}
Recalling the notation and the result of the nodal preservation in the proof of Theorem \ref{T:global1a},
since $v_j>0$ on $R^-$ and $v_j=0$ on $\partial R^-_l \cup \partial R^-_r$ for each $j$, by continuity, 
$\partial_q \varphi \geq 0$ on $R^-$ and  $\partial_q\varphi=0$ on $\partial R^-_l \cup \partial R^-_r$. 
Furthermore, since $\partial_q \varphi$ satisfies the second-order elliptic partial differential equation
\eqref{E:lambda0} and since $\partial_q\varphi \not \equiv 0$ in $R^-$, 
the maximum principle ensures that $\partial_q \varphi>0$ in~$R^-$.
Hence, $\partial_q\varphi$ may be expanded as a sine series
\[
\partial_q\varphi(q,p)=
\sum\limits^\infty\limits_{k=0} \varphi_k(p)\sin k(\pi/L)q,
\]
where $\varphi_k \to 0$ as $p \to -\infty$ for all $k$. Accordingly, \eqref{E:lambda0} is written as
\begin{align*}
\sum\limits^\infty\limits_{k=0}
\bigl((a^3\A\varphi'_k)'- k^2(\pi/L)^2a\A\varphi_k\bigr)\sin kq &=0
\qquad \text{in }\, R^-{,} \\
\sum\limits^\infty\limits_{k=0} \bigl(-2\sqrt{\lambda}\varphi'_k(0)+2g/\lambda\varphi_k(0)\bigr)\sin kq &=0
\qquad \text{on }\, T{.}
\end{align*}
%subject to the vanishing condition $\partial_q \phi_k \to 0$ as $ p\to -\infty$.
In particular, $\varphi_1$ solves the boundary value problem
\begin{alignat*}{2}
(a^3\A\varphi'_1)'&=(\pi/L)^2a\A\varphi_1 & &\quad \text{for }p \in (-\infty,0){,} \\
\lambda^{3/2}\varphi_1'(0)&=g\varphi_1(0){,} 
\end{alignat*}
subject to boundary conditions $\varphi_1, \varphi_1' \to 0$ as $p \to -\infty$. 
We observe that $\varphi_1$ is a solution of the Sturm-Liouville problem \eqref{E:Sturm} 
when $\epsilon=0$ with the generalized eigenvalue $\mu=-(\pi/L)^2$.

In view of the definitions  \eqref{D:R} and \eqref{D:Lambda} at $\epsilon=0$ 
it follows that $\Lambda^0(\lambda) \leq R^0(\varphi_1;\lambda)=-(\pi/L)^2$.
Suppose that $\Lambda^0(\lambda)<-(\pi/L)^2$;
the minimizer $\varphi_*$ of $\Lambda^0(\lambda)$ would be an eigenfunction 
corresponding to the simple eigenvalue $\Lambda^0(\lambda)$
(such that $R^0(\varphi_*;\lambda)=\Lambda^0(\lambda)$), and hence
$\varphi_*$ would not vanish on $p \in (-\infty,0)$. On the other hand,
\begin{equation*}\label{E:monotone}
\varphi_1(p) =\frac{2}{L}\int^0_{-L}\partial_q \varphi(q,p)\sin (\pi/L)q \;dq <0 
\qquad \text{for all }p \in (-\infty,0).
\end{equation*}
This contradicts the orthogonality
\[
\int^0_{-\infty} a\A \varphi_*(p)\varphi_1(p) dp =0.
\]
Therefore, $\Lambda^0(\lambda)=-(\pi/L)^2$, and 
$\lambda=\lambda^0$ follows from the monotonicity of $\Lambda^0$ (see \cite[Lemma~3.4]{CoSt04}).
This completes the proof.
\end{proof}

The global existence result for nontrivial solutions of \eqref{E:main} is now immediate 
and it is described in the next theorem.

\begin{theorem}[Global bifurcation for \eqref{E:main}]\label{T:global2}
For each $\delta>0$ let $\mathcal{C}'_\delta$ denote the maximal connected component 
of the closure in $\mathbb{R} \times X$ of $\Ss$ containing $(\lambda^0,0)$.

\renewcommand{\labelenumi}{{\rm(\roman{enumi})}}
\begin{enumerate}
\item The continuum $\mathcal{C}'_\delta$ either is unbounded in\/ $\mathbb{R} \times X$ 
or intersects $ \partial \OO$.
  
\item Each nontrivial solution lying on $\mathcal{C}'_\delta$
has the nodal properties \eqref{nodal1}-\eqref{nodal3}.
\end{enumerate}
\end{theorem}

\begin{proof}
(i) By virtue of Theorem~\ref{T:A6}, it suffices to show that
if $U$ is a bounded open set with $(\lambda^0,0) \in U$ and $\overline{U} \subset \OO$,
then $\partial U \cap \Ss \neq \emptyset$. 
Let $U$ be such an open set.

The result of Corollary~\ref{C:AT} says that there are sequences $\{\epsilon_j\}$ 
such that  $\epsilon_j \to 0$ as $j \to \infty$ with $0<\epsilon_j<1$ small, 
and $\{(\lambda_j,w_j)\} \subset \overline{U}$ such that
\[
(\lambda_j, w_j) \in \partial U \cap \mathcal{C}^{\epsilon_j}_\delta \qquad \text{for each $j$}.
\]
It is immediate that $(\lambda_j,w_j) \in \OO$ forms a bounded sequence in $\mathbb{R} \times X$ 
and that 
\begin{equation}\label{E:wj}
(F_1(\lambda_j, w_j) -\epsilon_j w_j, F_2(\lambda_j,w_j))=0 \qquad \text{for each $j$}.
\end{equation}
Possibly by relabeling, $\lambda_j \to \lambda$  as $j \to \infty$ for some $\lambda$. 
It will follow by the methods in the proof of Lemma~\ref{L:compact}
that $\{w_j\}$ has a subsequence which converges to $w$ in $X$. 
By continuity, $F(\lambda, w)=0$. 
%Recalling the definitions of $\mathcal{S}_\delta$ and $\mathcal{C}'_\delta$ 
%then, by continuity, $(\lambda, w) \in \partial U \cap \Ss$.
%For the rest of the proof it is assumed that $\lambda_j \to \lambda$ in $\mathbb{R}$ as $j \to \infty$.

The remainder of the proof is nearly identical to those of 
Lemmas~\ref{L:proper} or Lemma \ref{L:compact}, and we only outline its various steps.

Without loss of generality, we assume $|\lambda_j|+\|w_j\|_X<M$ for each $j$ for some $M>0$. 
The result of Corollary \ref{C:exp-decay} states that 
$\partial_q w_j(q,p)$ decays exponentially as $p \to -\infty$ in $(q,p) \in R$ uniformly for $j$. 
As is done in Lemma \ref{L:compact}, then, 
$\partial_q^2 w_j$ decay exponentially as $p \to -\infty$ uniformly for $j$,
and $w_j(0,\cdot)$ decays as $p \to -\infty$ uniformly for $j$.
In view of \eqref{E:w-decomp}, by adapting arguments of Ascoli type,   
it follows that $\{ w_j\}$ converges to $w$ as $j \to \infty$ in $C^0(\overline{R})$.

Next, since $\{ w_j\}$ is bounded in $X$, an interpolation inequality similar to that in 
\eqref{E:interp} (see \cite[Theorem 3.2.1]{Kry96}, for instance)
asserts that $w_j \to w$ as $j \to \infty$ in $C^3_{per}(\overline{R})$.
Since $A^P(\lambda_j,w_j)$ is uniformly elliptic and $B^P(\lambda_j,w_j)$ is uniformly oblique 
for each $j$, Schauder theory~\cite{ADN59} asserts that the inequality
\begin{align*}
\| w_j-w \|_X \leq C(\|A^P(\lambda_j,w_j)[w_j-w]\|_{Y_1}+
\|B^P(\lambda_j,w_j)[w_j-w]\|_{Y_2} +\| w_j-w\|_Z)
\end{align*}
holds for each $j$, where $C>0$ is independent of the index $j$.
It is straightforward that
\begin{align*}
A^P(\lambda_j,w_j)[w_j-w]=\epsilon_j (w_j-w)&-\bigl(A^P(\lambda_j,w_j)-A^P(\lambda,w)\bigr)w\\
&\quad -\bigl(f_1(\lambda_j,w_j)-f_1(\lambda,w)\bigr)+\epsilon_jw,\\
B^P(\lambda_j,w_j)[w_j-w] =\,
-\bigl(B^P(\lambda_j, &w_j)-B^P(\lambda,w)\bigr)w-\bigl(f_2(\lambda_j,w_j)-f_2(\lambda,w)\bigr).
\end{align*}
By arguments completely analogous to those in the proof of Lemma~\ref{L:proper}, then,
\[ 
\| A^P(\lambda_j,w_j)[w_j-w]\|_{Y_1} \to 0\qquad \text{as }j \to \infty,
\]
and
\[
\|B^P(\lambda_j,w_j)[w_j-w]\|_{Y_2} \to 0 \qquad \text{as } j \to \infty
\]
follow, whence it follows that $w_j \to w$ in~$X$. 

By continuity, $F(\lambda,w)=0$. If $w$ is not trivial then $(\lambda,w) \in \mathcal{C}'_\delta$,
upon recalling the definitions of $\mathcal{S}_\delta$ and $\mathcal{C}'_\delta$.
If $w=0$ then by Lemma \ref{L:nodal}, it must follows that $\lambda=\lambda^0$,
and hence $(\lambda, w) \in \mathcal{C}'_\delta$. This proves the assertion.

(ii) The proof, requiring to show the preservation of the nodal property of $\mathcal{N}$
along the continuum $\mathcal{C}'_\delta \setminus (\lambda^0,0)$, 
is nearly identical to that of \cite[Lemma C.3]{Hur06}. 
Suppose the contrary. Since $\Cd \subset \Ss$ is connected, 
there must be a nontrivial solution $(\lambda, w) \in \Cd$
with $w_q \not\equiv 0$ such that 
at least one of the nodal properties \eqref{nodal1}-\eqref{nodal3} would fail for $(\lambda, w)$.
We argue by contradiction using the maximum principle, 
the Hopf boundary lemma, and its sharp form at corner points due to Serrin.
The detail of the proof is in \cite[Appendix C]{Hur06}.
This completes the proof.
\end{proof}

The proof of (i) in the above theorem entails to take the limit of 
a bounded set in $\mathcal{C}^\epsilon_{\delta}$ as $\epsilon \to 0$ in the sense to show that 
a bounded sequence $\{(\lambda_j, w_j)\}$ for each $(\lambda_j,w_j)$ is a solution 
of the approximate problems \eqref{E:approx} with $\epsilon=\epsilon_j$ 
converges to a solution of the singular problem \eqref{E:main}
as $j \to \infty$ if $\epsilon_j \to 0$ as $j\to \infty$. 
The uniform decay property of solutions of \eqref{E:main} ensures that 
the convergence takes place in the strong topology of $\mathbb{R} \times X$.

The purpose of the following lemma is 
to obtain bounds for the higher derivatives of $w$
in terms of $w$ and $w_p$ uniformly along the continuum~$\Cd$.

\begin{lemma}\label{T:regularity}
For each $\delta>0$, if $\sup_{(\lambda,w) \in \mathcal{C}'_\delta} \lambda <\infty$ then
$\sup_{(\lambda, w) \in \Cd} (\|w\|_{C^0(\overline{R})}+\|w_p\|_{C^0(\overline{R})}) <\infty$
implies $\sup_{(\lambda, w)\in\Cd}\| w\|_X<\infty$.
\end{lemma}

\begin{proof}
Provided that $\lambda$ along $\mathcal{C}'_\delta$ is bounded, 
by the maximum principle it follows that 
$w_q$ along $\Cd$ is bounded by the maxima in $T$ of $w$ and $w_p$ in $\Cd$.
Then, by a priori estimates of Schauder type due to Lieberman and Trudinger \cite{LiTr}
for quasilinear elliptic partial differential equations with nonlinear oblique boundary conditions, 
it follows that the higher derivatives of $w$ along $\Cd$ are bounded by 
the maximum norms in $\overline{R}$ of $w_p$ and $w_q$ along $\Cd$. 
The detail of the proof is in \cite[Section 5.2]{Hur06}.
\end{proof}

\begin{remark}\label{R: alt}\rm
By virtue of Lemma \ref{T:regularity}, in case $\Cd$ is unbounded in $\mathbb{R} \times X$, 
there is a sequence of solution pairs 
$\{(\lambda_j,w_j)\} \subset \Cd$ such that either
\begin{enumerate}
\item ${\displaystyle \lim_{j \to \infty} \lambda_j=\infty}$ or
\item $\lambda_j$ is bounded for all j while either 
${\displaystyle \lim_{j \to \infty} \| w_j \|_{C^0(\overline{R})}=\infty}$ or 
${\displaystyle \lim_{j \to \infty} \| \partial_pw_j \|_{C^0(\overline{R})}=\infty}$.
\end{enumerate}
If the other alternative that $\Cd$ intersects $\partial \OO$ realizes, 
there is a solution pair $(\lambda, w) \in \Cd$ such that 
one of the following holds:  
\begin{enumerate}\addtocounter{enumi}{2}
\item $\lambda = 2\Gamma_{\inf} + \delta$, 
\item $a^{-1}\A+w_p =\delta$  somewhere in $\overline{R}$, 
\item $ w=\frac{2\lambda-\delta}{4g}$ somewhere on $T$.
\end{enumerate}
\end{remark}

%%%%%%%%%%%%%%%%%%%%%%%%%%%%%%%%%%%%%%%%%%%%%%%%%%%%%%%%%%%%%%%%%%%%%%%%%%%%%%%%%%%%%%%%%%%%%%%%%%%%%%%%%%%%%%%%%%%%%%%%%%%%%%%%%%%%%%%%%%%%%%%%%%%%%%%%%
\section{Properties of rotational Stokes waves}\label{S:properties}
%%%%%%%%%%%%%%%%%%%%%%%%%%%%%%%%%%%%%%%%%%%%%%%%%%%%%%%%%%%%%%%%%%%%%%%%%%%%%%%%%%%%%%%%%%%%%%%%%%%%%%%%%%%%%%%%%%%%%%%%%%%%%%%%%%%%%%%%%%%%%%%%%%%%%%%%%
Established are properties of solutions of \eqref{E:main}, and in turn, solutions of \eqref{E:stream}. 
The main results are proved.

%%%%%%%%%%%%%%%%%%%%%%%%%%%%%%%%%%%%%%%%%%%%%%%%%%%%%%%%%%%%%%%%%%%%%%%%%%%%%%%%%%%%%%%%%%%%%%%%%%%%%%
\subsection{Properties of rotational Stokes waves}\label{SS:properties}
Throughout this subsection, $\delta>0$ is held fixed. 
For each solution pair $(\lambda, w) \in \mathcal{C}'_\delta$ of \eqref{E:main}, 
let $(c, \eta, \psi)$ be the corresponding solution triple of \eqref{E:stream}. 
As such, $(\lambda,w)$ and $(c,\eta,\psi)$ are related 
via the transforms in Section \ref{SS:reformulation}. More precisely, 
\begin{alignat}{3}
c^2&=\lambda-2\Gamma_\infty,\qquad  &&\eta(x)=w(x,0)-\frac{\lambda}{2g} ,\label{E:c-eta} \\
\psi_x&=-\frac{w_q}{a^{-1}(\lambda)+w_p}, \qquad && \psi_y=-\frac{1}{a^{-1}\A +w_p}. \label{E:velocity}
\end{alignat}
Related to these, 
\begin{equation}\label{E:hw}
h_q=w_q=-\frac{\psi_x}{\psi_y}, \qquad h_p=a^{-1}\A +w_p=-\frac{1}{\psi_y}.
\end{equation}

Recalled are the notations 
\[ \Omega_\eta=\{(x,y): -\infty<x<\infty, \, -\infty<y<\eta(x)\}, \quad 
S_\eta=\{(x,\eta(x)): -\infty<x<\infty\}.\]
Since $(\lambda, w) \in \mathbb{R} \times X$ it follows that 
$(c,\eta,\psi) \in \mathbb{R}_+\times C^{3+\alpha}_{per}(\mathbb{R}) \times 
C^{3+\alpha}_{per}(\overline{\Omega}_\eta)$.

Established properties are summarized of solutions of \eqref{E:main}, and in turn, 
solutions of \eqref{E:stream}, and further properties are inferred. 
These results are of independent interests.

First, by the nodal properties of $w$, for any nontrivial solution of \eqref{E:stream}
\begin{alignat}{3}\label{property1} 
\eta_x&(x)>0 \qquad &&\text{for}\quad  -L<x<0, \\
 \psi_x&(x,y)>0 \qquad && \text{for}\quad (x,y) \in \Omega_\eta \cup S_\eta, \quad-L<x<0, 
\end{alignat}
and 
\begin{equation}\label{property2}
\psi_{xx}(\pm L,y)>0 \quad \text{for $(\pm L,y) \in \Omega_\eta$}, \qquad 
\psi_{xx}(0,y)<0 \quad \text{for $(0,y) \in \Omega_\eta$}.
\end{equation}
hold. By evenness and periodicity properties of $w$, furthermore, 
\begin{alignat}{3}\label{property3}
\eta_x&(0)=\eta_x(\pm L)=0, \\
\psi_x&(0,y)=\psi_x( \pm L)=0 \qquad \text{for}\quad (0,y), (\pm L,y) \in \Omega_\eta
\end{alignat}
hold for any solution of \eqref{E:stream}.

By Lemma \ref{L:exp-decay}, any bounded solution of \eqref{E:stream} enjoys 
the following exponential decay estimate 
\begin{equation}\label{property4}
|\psi_x(x,y)|<C' e^{\sigma' y} \qquad \text{for $(x,y) \in \Omega_\eta$},
\end{equation}
where $C'>0$ large and $\sigma'>0$ depend only on $\delta>0$, $c>0$ and 
the H\"older norms of $\eta$ and $\psi$. 
%By symmetry, $\psi_x(x,y)$ for $0<x<L$ enjoys an analogous exponential decay as $y \to -\infty$. 
By the classical gradient estimate in the elliptic theory \cite[p. 37]{GiTr01} it follows that 
\begin{equation}\label{property5}
|\psi_{xx}(0,y)| \leq C''e^{\sigma' y} \qquad \text{as $y \to -\infty$\quad uniformly for $x$.}
\end{equation}
%Since $-\Delta \psi=\gamma(\psi)$, subsequently, $\psi_y \to -c$ and 
%$\psi_{yy} \to 0$ as $y \to -\infty$ at least as fast as $\gamma(\psi) \to 0$ as $y \to -\infty$. 

Next, the maximum principle is employed to yield bounds for 
the relative velocity $\nabla \psi$ and the pressure. 

\begin{lemma}[Bounds for the velocity]\label{L:velocity-bound}
A nontrivial solution pair $\eta(x)$ and $\psi(x,y)$ of \eqref{E:stream} satisfies
\begin{equation}\label{E:velocity-bound1}
\psi_y^2(0,\eta(0))\leq |\nabla \psi(x,y)|^2-2\Gamma(-\psi(x,y)) \leq  \psi_y^2(\pm L,\eta(\pm L))
\end{equation}
for any $(x,y) \in \overline{\Omega}_\eta$. 
\end{lemma}

The result holds without any restriction on the vorticity function $\gamma$. 
The equality holds if the free surface is trivial. 

\begin{proof}
The proof is an immediate application of the maximum principle due to Sperb \cite[Section 5.2]{Spe}.
See also \cite[Theorem 3.1]{Var09}. Here, we include the proof for the sake of completeness. 

Let us define a function $W: \overline{\Omega}_\eta \to \mathbb{R}$ by
\begin{equation}\label{D:W}
W(x,y)=\frac{1}{2} |\nabla \psi(x,y)|^2-\Gamma(-\psi(x,y)).
\end{equation}
It is straightforward that 
\[ \Delta W+\frac{L_1}{|\nabla \psi|^2}W_x +\frac{L_2}{|\nabla \psi|^2} W_y=0
\qquad \text{in }\, \Omega_\eta,\]
where $L_1=2\gamma(\psi)\psi_x-2W_x$ and $L_2=2\gamma(\psi)\psi_y-2W_y$. Since 
\[
W_y=\psi_x\psi_{xy}+\psi_y\psi_{yy}+\gamma(\psi)\psi_y 
=\psi_x\psi_{xy}-\psi_y\psi_{xx} \to 0 \qquad \text{as $y \to -\infty$,}
\]
by the maximum principle and the Hopf boundary lemma, it follows that 
\[ \min_{S_\eta} W(x,y) \leq W(x,y)\leq \max_{S_\eta} W(x,y) 
\qquad (x,y) \in \overline{\Omega}_\eta.\]
The assertion then follows since $\Gamma(\psi(x,y))=0$ on $S_\eta$ and 
since $|\nabla\psi(x,\eta(x))|^2=-2g\eta(x)$ is 
nonincreasing for $-L\leq x\leq 0$ and is nondecreasing for $0\leq x\leq L$.
\end{proof}

Let us define another function $B: \overline{\Omega}_\eta \to \mathbb{R}$ by
\begin{equation}\label{D:B}
B(x,y)=\frac{1}{2}|\nabla \psi(x,y)|^2+gy-\Gamma(-\psi(x,y)).
\end{equation}
Note that $B(x,y)$ is the negative of the hydrostatic pressure up to a constant. 

\begin{lemma}[Pressure estimates]\label{L:-pressure}
For any solution pair $\eta(x)$ and $\psi(x,y)$ of \eqref{E:stream} the following inequality 
\begin{equation}\label{E:pressure1} 
B(x,y)-\frac{1}{2} \max\Big(0, \sup_{0\leq \psi<\infty} \gamma(\psi)\Big)\psi \leq 0
\qquad \text{in}\quad \Omega_\eta
\end{equation}
holds. If, in addition, $\eta(x)$ and $\psi(x,y)$ satisfy 
\begin{equation}\label{C:negative}
g+\gamma(\psi)\psi_y \geq0 \qquad \text{in}\quad \Omega_\eta,
\end{equation}
then 
\begin{equation}\label{E:-pressure}
B(x,y)\leq 0 \quad \text{in $\Omega_\eta$}\quad \text{and} \quad 
\frac{\partial B}{\partial n}(x,y) >0 \quad \text{on $S_\eta$,}
\end{equation} 
where $\partial/\partial n$ denotes the outward normal derivative at $S_\eta$.
\end{lemma}

The condition \eqref{C:negative} is valid in the irrotational setting and for non-positive vorticities.

\begin{proof}The proof is similar to that of Lemma \ref{L:velocity-bound}.
As in the proof of Lemma \ref{L:velocity-bound}, it is straightforward that 
\[ \Delta B+\frac{L_1}{|\nabla \psi|^2}B_x +\frac{L_2+2g}{|\nabla \psi|^2} B_y=
\frac{2g}{|\nabla \psi|^2} (g+\gamma(\psi)\psi_y)\geq 0 \qquad \text{in }\, \Omega_\eta\]
under the hypothesis \eqref{C:negative},
where $L_1$ and $L_2$ are given in the course of the proof of Lemma \ref{L:velocity-bound}.
By the Bernoulli equation (\ref{E:stream}c), it follows that $B(x,y)=0$ on $S_\eta$.
Since 
\[B_y=\psi_x\psi_{xy}-\psi_y\psi_{xx} +g \to g \qquad \text{as $y \to -\infty$,}\]
the assertions follow by the maximum principle and the Hopf boundary lemma.

Repeating the argument in the proof, for an arbitrary vorticity asserts \eqref{E:pressure1}.
\end{proof}

Under the condition \eqref{C:negative}, since $B(x,y)=0$ on the free surface $S_\eta$ it follows that 
\[ \Big[(1,\eta_x) \cdot (B_x, B_y) \Big]_{S_\eta} =0,\]
while the result of Lemma \ref{L:-pressure} says that
\[ \Big[(-\eta_x,1) \cdot (B_x,B_y) \Big]_{S_\eta}>0.\]
Thus, $B_x(x,\eta(x))<0$ for $-L<x<0$ and $B_x(x,\eta(x))>0$ for $0<x<L$.  
On the other hand, for any vorticity function $\gamma$, it follows that 
\begin{align*}
\frac{d}{dx}\left(\frac{1}{2} \psi^2_y(x,\eta(x))\right)&=
\psi_y\psi_{xy} +\psi_y\psi_{yy} \eta_x \\
&=\psi_y\psi_{xy}-\psi_x\psi_{yy} =B_x(x,\eta(x)).
\end{align*}
Indeed, $\psi_x+\eta_x\psi_y=0$ at the free surface $S_\eta$.
Since $\psi_y(x,\eta(x))<0$, it follows that
$\psi_y(x,\eta(x))$ is nonincreasing for $-L<x<0$ and nondecreasing for $0<x<L$.
We summarize this result. 

\begin{lemma}\label{L:monotone}
Under the condition \eqref{C:negative}, any solution pair $(\eta(x),\psi(x,y))$ of \eqref{E:stream} satisfies 
\begin{equation}\label{E:monotone}
\psi_y(-L, \eta(-L)) \leq \psi_y (x,\eta(x)) \leq \psi_y(0,\eta(0))<0 \qquad \text{for}\quad -L\leq x\leq 0.
\end{equation}
\end{lemma} 

If the vorticity is non-negative and monotone with depth, 
then the following alternative bound is available for the pressure. 

\begin{lemma}[Pressure estimate for positive vorticities]\label{L:+pressure}
If $\gamma(r) \geq 0$ and $\gamma'(r)  \leq 0$ for $0\leq r<\infty$ then
\begin{equation}\label{E:+pressure}
B(x,y) + \Gamma(-\psi(x,y)) \leq 0 \qquad \text{in}\quad \Omega_\eta.
\end{equation}
\end{lemma}

\begin{proof}
Note that \[B(x,y)+\Gamma(-\psi(x,y))=\frac{1}{2}|\nabla \psi(x,y)|^2+gy.\] 
The (steady) Euler equations yield
\[ B_x=-\psi_x\psi_{yy}+\psi_y\psi_{xy}\quad \text{and} \quad B_y=-\psi_y\psi_{xx}+\psi_x\psi_{xy}+g,\] 
whence  
\[ \Delta (B+\Gamma(-\psi)) +2\gamma'(\psi)(B+\Gamma(-\psi))
=\psi_{xx}^2+\psi_{yy}^2+2\psi_{xy}^2 \leq 0 \qquad \text{in }\, \Omega_\eta.\]
Since $\gamma'(\psi)\leq 0$, the maximum principle asserts that $B+\Gamma(-\psi)$ 
attains its maximum in $\overline{\Omega}_\eta$ either on the free surface or at the infinite bottom.
On the other hand, since
\[ \partial_y (B-\Gamma(-\psi))=-\psi_y\psi_{xx}+\psi_x\psi_{xy}+g-\gamma(\psi)\psi_y \to g
\qquad \text{as $y \to -\infty$},\]
the maximum of $B-\Gamma(-\psi)$ is attained at the free surface. 
The assertion then follows since $\Gamma(-\psi)=0$ on the free surface. 
\end{proof}

As is done for the vorticities which satisfies \eqref{C:negative}, in the setting of Lemma \ref{L:+pressure}, 
it follows that 
\[ \Big[(1,\eta_x) \cdot (B_x , B_y) \Big]_{S_\eta} =0,\]
and 
\begin{equation}\label{E:n(x)} 
n(x):=\Big[(-\eta_x,1) \cdot (B_x,B_y) +\gamma(-\eta_x,1) \cdot (\psi_x,\psi_y) \Big]_{S_\eta}>0.
\end{equation}
%Therefore, $B_x<-\gamma \eta_x \psi_y$ on $S_\eta$ for $-L< x< 0$ 
%and $B_x > -\gamma \eta_x \psi_y$ on $S_\eta$ for $0<x<L$. 
Unfortunately, this does not yield the monotonicity of $\psi_y(x,\eta(x))$. Instead, 
\begin{equation}\label{E:old-smallness}
\begin{split}
\frac{d}{dx}\psi_y^2(x,\eta(x)) =&2\psi_y\psi_{xy}+2\psi_x(\psi_{xx}+\gamma(0))\\
=&-\eta_x\left( \frac{n(x)}{1+\eta_x^2(x)} +\gamma(0)\psi_y(x,\eta(x))\right).
\end{split}
\end{equation}

If the vorticity is non-positive and monotone with depth, then 
a simple maximum principle yields bounds for $\psi_y$,
a stronger result than that in Lemma \ref{L:velocity-bound}.

\begin{lemma}\label{L:min-max}
If $\gamma(r)\leq 0$ and $\gamma'(r)\geq 0$ for $r \in [0,\infty)$, 
any nontrivial solution pair $(\eta(x),\psi(x,y))$ of \eqref{E:stream}
satisfies
\begin{equation}\label{E:velocity-bound2} 
\psi_y(0,\eta(0)) \leq \psi_y(x,y) \leq \psi_y(\pm L,\eta(\pm L))
\end{equation}
for any $(x,y) \in \Omega_\eta$.
\end{lemma}

\begin{proof}
By differentiating (\ref{E:stream}b) in the $y$-variable, we obtain
\[\Delta \psi_y=-\gamma'(\psi)\psi_y \qquad \text{in}\quad \Omega_\eta.\]
Since $\psi_y<0$ in $\Omega_\eta$ and $\gamma'(s)\geq 0$,
by the maximum principle, $\psi_y$ cannot have an interior maximum in $\overline{\Omega}_\eta$. 
Moreover, since $\psi_y \to c<0$ as $y \to -\infty$, the maximum of $\psi_y$ must be on $S_\eta$. 
The assertion then follows by \eqref{E:monotone}.
\end{proof}

If the vorticity is non-positive, 
then the amplitude is bounded by the speed of wave propagation.

\begin{lemma}
If $\gamma (r)\leq 0$ and if $(\eta(x),\psi(x,y))$ is a solution pair to \eqref{E:stream} 
with the parameter $c>0$, then 
\begin{equation}\label{E:c-bound}
0\leq (2g)^{3/2}(|\eta(\pm L)|^{3/2}-|\eta(0)|^{3/2})= 
|\psi_y(\pm L,\eta(\pm L))|^{3}-|\psi_y(0,\eta(0))|^{3}\leq cL.
\end{equation}
\end{lemma}

\begin{proof}
By integrating (\ref{E:stream}b) in the domain $\Omega^-_\eta=\{(x,y) \in \Omega_\eta : -L<x<0 \}$ 
and in the light of the Green's theorem, we obtain that 
\begin{align*}
\iint_{\Omega^-_{\eta}} -\gamma(\psi) \, dxdy =&  \iint_{\Omega^-_{\eta}} \Delta \psi \, dxdy \\
=& -\int_{S^-_{\eta}} \frac{\partial \psi}{\partial n} dl 
+ \int^0_{-\infty} \psi_x(-L,y) dy  \\ &+\lim_{Y \to -\infty} \int^0_{-L} -\psi_y(x,Y)dx 
-\int^0_{-\infty} \psi_x (0,y) dy, \\
=&  -\int_{S^-_{\eta}} \frac{\partial \psi}{\partial n} dl+cL,
\end{align*}
where $S^-_\eta$ is the top boundary of $D^-_\eta$, $\partial/\partial n$ denotes 
the outward normal derivative at $S^-_\eta$, and $dl$ means the line integration along $S^-_\eta$.
The second equality uses the evenness and periodicity of $\psi$ in the $x$-variable
and the last equality uses that $\psi_y(x,y)\to -c$ as $y\to -\infty$. 
Since $\gamma(\psi)\leq 0$ we further obtain
\begin{equation}\label{E:green's}
-\int_{S^-_{\eta}} \frac{\partial \psi}{\partial n} dl + cL\geq 0.\end{equation}

On the other hand, since $\psi(x,y)=0$ on $S^-_\eta$, it follows that 
\[ \frac{\partial \psi}{\partial n}=(1+\eta_x^2(x))\psi_y=- (1+\eta_x^2(x))^{1/2} |\nabla \psi|
\qquad \text{on}\quad S^-_\eta.\]
Moreover, since $(1+\eta_x^2(x))^{1/2} \geq \eta_x(x)$ for $-L\leq x\leq 0$, 
the inequality \eqref{E:green's} yields that 
\[\int^0_{-L} \eta_x(x) |\nabla \psi(x,\eta(x))| dx \leq cL.\]
Finally, upon substituting of $|\nabla \psi|$ by the Bernoulli's equation (\ref{E:stream}c) 
and integrating the above, we obtain
\[ (-2g\eta(-L))^{3/2}-(-2g\eta(0))^{3/2} \leq cL.\]
The assertion then follows by the use of the Bernoulli's equation again. 
\end{proof}

Finally, the relative flow speeds at the crest and at the trough are bounded by 
the upstream speed of the underlying shear flow, and in turn, by the speed of wave propagation.

\begin{lemma}[Relative flow speed at the crest]\label{L:speed}
Any nontrivial solution pair $(\lambda, w) \in \mathcal{C}'_\delta$ of \eqref{E:main} satisfies 
\[(\lambda^{-1/2}+w_p(0,0))^{-2}<\lambda.\] 
Hence, the solution triple $(c,\eta,\psi)$ corresponding to$ (\lambda, w)$ satisfies
\begin{equation}\label{crest}
 \psi_y^2(0,\eta(0))<\lambda= c^2+2\Gamma_\infty.
\end{equation}
\end{lemma}

\begin{proof}
Since $w$ belongs to the nodal set $\mathcal{N}$, it follows that $w_q=0$ and $w_{qq}<0$
on the half-line $q=0$ and $-\infty<p<0$. 
Recalling \eqref{D:w}, subsequently, $h_q=0$ and $h_{qq}<0$ on $q=0$ and $-\infty<p<0$.
Thus, (\ref{h-prob}a) along the half-line $q=0$ and $-\infty<p<0$ reduces to the inequality
$h_{pp}>-\gamma(-p)h^3_p$. Further,
\[
-\left( \frac{1}{h^2_p(0,p)}\right)' >-2\gamma(-p) \qquad \text{for } -\infty<p< 0.
\]
Recalling that $h_p \to 1/c$ as $p \to -\infty$, integration of the above inequality over $(-\infty,0)$ 
then yields 
\[ h_p^{-2}(0,0) -c^2 < 2\Gamma_\infty.\]
The assertion then follows by the definition $\lambda=c^2+2\Gamma_\infty$.
\end{proof}

The same calculation carried out on the half-line $q=\pm L$ and $p \in (-\infty,0)$ 
leads to an analogous  bound for the relative flow speed at the wave trough
\begin{equation}\label{trough}
 \psi_y^2( \pm L,\eta( \pm L)) > \lambda=c^2+2\Gamma_\infty.
\end{equation}

%%%%%%%%%%%%%%%%%%%%%%%%%%%%%%%%%%%%%%%%%%%%%%%%%%%%%%%%%%%%%%%%%%%%%%%%%%%%%%%%%%%%%%%%%%%%%%%%%%%%%%
\subsection{Proof of the main results}\label{S:proof}

%Since$h(x,0)=\eta(x)-B/2g$ and $h(x,0)=w(x,0)-\lambda/2g$, the above estimates
%\eqref{crest} and \eqref{trough} together with the Bernoulli condition (\ref{E:stream}d) yield that 
%\begin{alignat}{2}
%w(0,0)&=\,\, h(0,0)\,\,-&&\, H(0)=\frac{\lambda- \psi_y^2(0,\eta(0)) }{2g} >0{,} \label{w-crest}\\
%w(\pm L,0)&=h( \pm L,0)\,-\,&&\, H(0)
%=\frac{\lambda-\psi_y^2(\pm L,\eta(\pm L)) }{2g}<0. \label{w-trough}
%\end{alignat}
Let $\mathcal{C}_\delta$ be the connected set 
in $\mathbb{R}_+\times C^{3+\alpha}(\mathbb{R}) \times C^{3+\alpha}(\overline{\Omega}_\eta)$ 
of solution triples $(c,\eta,\psi)$ of \eqref{E:stream},
corresponding to the continuum $\Cd$ of solutions $(\lambda, w)$ of \eqref{E:main}
via the transforms \eqref{E:c-eta} and \eqref{E:velocity}.

By construction, $\Cd \subset \mathcal{C}'_{\delta'}$ if $\delta>\delta'$.
Let \[\mathcal{C}' =\bigcup_{\delta>0} \Cd,\] and let $\mathcal{C}=\sup_{\delta>0} C_\delta$. 

For each $\delta>0$, by virtue of Theorem~\ref{T:global2} and Remark \ref{R: alt}, 
at least one of the following holds:
\begin{enumerate}
\item there exists a sequence $\{(\lambda_j,w_j)\} \subset \Cd$ such that
$\lim_{j \to \infty} \lambda_j = \infty$;
\item there exists a sequence $\{(\lambda_j,w_j)\} \subset \Cd$ such that
$\lambda_j$ are bounded but $\sup_{\overline{R}}|w_j| \to \infty$;
\item there exists a sequence $\{(\lambda_j,w_j)\} \subset \Cd$ such that
$\lambda_j$ are bounded yet $\sup_{\overline{R}}|\partial_p w_j| \to \infty$;
\item there exists a solution pair $(\lambda,w) \in \Cd$ such that
$\lambda+2\Gamma_{\inf} =\delta$;
\item there exists a solution pair $(\lambda,w) \in \Cd$ such that
$a^{-1}\A+w_p =\delta$ at some point in $\overline{R}$;
\item there exists a solution pair $(\lambda,w) \in \Cd$ such that
$w=\frac{2\lambda-\delta}{4g}$
at some point on~$T$.
\end{enumerate}

\begin{proof}[Proof of Theorem \ref{T:main}]
Due to \eqref{E:c-eta}, 
Alternative (1) implies that there exists a sequence $\{(c_j, \eta_j, \psi_j)\} \subset \mathcal{C}_\delta$
for which $\lim_{j \to \infty} c_j=\infty$. 
This corresponds the first alternative in (ii) of Theorem \ref{T:main}. 

Our task is to give an interpretation of each alternative (2) through~(6) 
in terms of the traveling speed $c$ or the relative flow speed $\psi_y$ 
to prove assertion~(ii) of Theorem~\ref{T:main}.

\

{\it Alternative} (2). For each $j$ and for each $-\infty< p \leq 0$, by the nodal configuration of $w$
it follows that $\partial_q w_j(q,p)>0$ for $-L<q<0$ 
whereas by oddness of $\partial_q w_j$ it follows that $w_q(q,p)<0$ for $ 0<q<L$.
Thus, $w(q,p)$ attains in $\overline{R}$ its maximum somewhere along the line $q=0$ 
and its minimum somewhere along the line $q=\pm L$. 
Therefore, this alternative implies that there exists a sequence $\{p_j\}$ in the interval $(-\infty,0]$ 
such that \[\text{either $\lim_{j\to \infty} w_j(0,p_j)=\infty$
\quad or \quad $\lim_{j \to \infty} w_j(\pm L,p_j)= -\infty$.}\]

Suppose $\lim_{j\to \infty}  w_j(0,p_j)=\infty$. Since $w_j(q,p) \to 0$ as $p \to -\infty$ for each $j$,
furthermore, $\{p_j\}$ is bounded below. That is, there exists $-\infty<p_0<0$ such that 
\[\sup_{-\infty<p\leq 0} w_j(0,p)=\sup_{p_0\leq p\leq 0}w_j(0,p).\]

On the other hand, for $p_0\leq p\leq 0$,
\[ \sup_p w_j(0,p)=\int^0_{p_0} \partial_p w_j(0,p')dp' \leq |p_0| \sup_p |\partial_p w_j(0,p)|.\]
Thus, $\lim_{j\to \infty} \sup_p w_j(0,p) =\infty$ implies that 
$\lim_{j\to \infty} \sup_p \partial_pw_j(0,p) =\infty$. 
That means, Alternative (3) must occur, as well. 
Below, we will show that in the case of Alternative (3),
\[ \lim_{j \to \infty} \inf_{\Omega_{\eta_j}} \partial_y \psi_j(x,y) = 0\]
holds, where  $(c_j,\eta_j, \psi_j)$ corresponds to $(\lambda_j,w_j)$ 
via \eqref{E:c-eta} and \eqref{E:velocity}.
The case $\lim_{j \to \infty} \inf_{-\infty<p<0} w_j(\pm L,p)= -\infty$ is treated similarly.

\

{\it Alternative} (3). Since $\lambda_j +2\Gamma_{\inf} >\delta$ 
and $a^{-1}(\lambda_j) +\partial_pw_j>\delta$ for all $j$, 
it follows that $a(p;\lambda_j) >\delta^{1/2}$ for all $j$, and consequently, 
\[\delta < a^{-1}(\lambda_j) + \partial_p w_j <  \delta^{-1/2} +\partial_p w_j\]
for all $j$, that is, $\delta -\delta^{-1/2}<\partial_p w_j$ for all $j$.
Therefore, $\lim_{j \to \infty} \sup_{\overline{R}}\partial_p w_j(q,p)= \infty$ must hold. 
By \eqref{E:velocity} and \eqref{E:hw} then 
it is readily seen that \[ \lim_{j \to \infty} \inf_{\Omega_{\eta_j}} \partial_y \psi_j(x,y) = 0,\]
where $(c_j,\eta_j, \psi_j)$ corresponds to $(\lambda_j,w_j)$ via \eqref{E:c-eta} and \eqref{E:velocity}.

\

{\it Alternative} (4). Let us choose a sequence $\{\delta_j\}$ with $\delta_j \to 0+$ as $j \to \infty$ and 
to each $\delta_j$ choose $(\lambda_j, w_j) \in \Cdj$ such that $\lambda_j +2\Gamma_{\inf} =\delta_j$. 
We may assume that $\sup_{\overline{R}} \partial_p w_j<\infty$ for all $j$; 
otherwise, $\lim_{j \to \infty} \inf_{\Omega_{\eta_j}} \partial_y \psi_j(x,y) = 0$ must hold
by the treatment for Alternative~(3), 
where $(c_j,\eta_j, \psi_j)$ corresponds to $(\lambda_j,w_j)$ via \eqref{E:c-eta} and \eqref{E:velocity}.

Let us choose a sequence $\{p_j\}$, $-\infty<p_j \leq 0$, such that 
$2\Gamma (p_j) <2\Gamma_{\inf}+\delta_j$ for each $j$. 
It is straightforward that \[a(p_j;\lambda_j) =(\lambda_j +2\Gamma(p_j))^{1/2} <(2\delta_j)^{1/2},\]
whence 
\[
\partial_p h_j(q,p_j)=a^{-1}(p_j;\lambda_j)  + \partial_p w_j(q,p_j) 
\geq \delta_j^{-1/2} +\partial_p w_j(q,p_j).
\]
Snce $\partial_p w_j(q,p)$ is bounded, the right side increases unboundedly as $j \to \infty$. 
Accordingly, 
$\lim_{j \to \infty} \sup_{\overline{R}}\partial_p h_j (q,p)= \infty$, and 
 \[ \lim_{j \to \infty} \inf_{\Omega_{\eta_j}} \partial_y \psi_j(x,y) = 0,\]
 where $(c_j,\eta_j, \psi_j)$ corresponds to $(\lambda_j,w_j)$ via \eqref{E:c-eta} and \eqref{E:velocity}.

\
 
{\it Alternative} (5). Let us choose sequences $\{\delta_j\}$ 
and $\{(\lambda_j, w_j)\}$, $\{(q_j,p_j)\}$ such that 
$\delta_j \to 0+$ as $j \to \infty$ and $(\lambda_j,w_j) \in \Cdj$, $(q_j,p_j) \in \overline{R}$ with 
\[\partial_p h_j(q_j,p_j)=a^{-1}(p_j;\lambda_j) +\partial_p w_j (q_j,p_j)=\delta_j.\]
We may assume that $p_j$ is bounded below. 
Suppose on the contrary that $p_j \to -\infty$ as $j \to \infty$. 
Since $\partial_p w_j(q,p) \to 0$ as $p \to -\infty$, it implies that 
\[ \partial_p h_j(q_j,p_j) \to \frac{1}{(\lambda_j +2\Gamma(p_j) )^{1/2}}\] 
and the right side tends to zero as $j \to \infty$. 
It, in turn, implies $\lambda_j \to \infty$ as $j \to \infty$. 
We may also assume that $\sup_{T}|w_j(q,p)|$ is bounded. Otherwise, 
by the treatment of Alternative (2), 
$\lim_{j \to \infty} \sup_{\overline{\Omega}_{\eta_j}} \partial_y\psi_j (x,y) =0$ holds.

Let $\{(c_j, \eta_j, \psi_j)\}$ be the solution triples of \eqref{E:stream} 
corresponding to $\{(\lambda_j,w_j)\}$ via \eqref{E:c-eta} and \eqref{E:velocity}.
It is readily seen that $\partial_y \psi_j(x_j,y_j) =-\infty$ as $j \to \infty$,
where $x_j=q_j$ and $y_j(x, p_j)$ is the inverse of $\psi_j(x,y_j)$. 
Under the assumption that $p_0\leq p_j \leq 0$ for some $-\infty<p_0\leq 0$, 
we claim that $\lim_{j\to \infty}\sup_{p_0\leq p\leq 0} \partial_p h_j(q,p) =\infty$,
and correspondingly,
\[ \lim_{j \to \infty} \sup_{\overline{\Omega}_{\eta_j}} \partial_y\psi_j (x,y) =0.\]
Suppose that $\partial_p h_j(q,p)$ is bounded in $p_0\leq p\leq 0$ for all $j$. 
Since \[ \int^0_{p_0} \partial_p h_j(q,p')dp'=h_j(q,0)-h_j(q,p_0)=\eta_j(x)-y_j(x,p_0)<\infty,\]
and since $\eta_j(x)$ is bounded for $j$, it follows that $y_j(x,p_0)$ is bounded,
say $y_0<y_j(x,p_0)$ for some $y_0$. 
On the other hand, the pressure estimate \eqref{E:pressure1} yields that 
\[ \frac12|\nabla\psi_j (x,y)|^2+gy-\Gamma(-\psi_j(x,y)) -\frac12 \max(0, \sup \gamma(\psi_j)) \psi_j \leq 0\]
in $\Omega_{\eta_j}$. As $j \to \infty$. However, 
$\inf_{y_0\leq y\leq \eta_j(x)} \partial_y\psi_j (x,y) =-\infty$ 
while all the other terms except for the first are bounded for $j$. 
A contradiction therefore proves the claim.

\

{\it Alternative} (6). As is done for Alternative (5), let us choose sequences $\{\delta_j\}$ and 
$\{\lambda_j, w\}$ such that $\delta_j \to 0$ as $j \to \infty$ and 
$(\lambda_j,w_j)\in \Cdj$ such that $\lambda_j -2gw_j=\frac{1}{2} \delta_j$ somewhere on $T$. 
The nonlinear boundary condition $F_2(\lambda,w)=0$ then yields
\[
(\lambda_j^{-1/2}+\partial_p w_j)^2  =
\frac{1+\partial_q w_j^2}{\lambda_j -2gw_j}= \frac{2}{\delta_j}
\]
somewhere on $T$. As $j \to \infty$, it implies that 
\[ \lim_{j \to \infty} \sup_{S_{\eta_j}} \partial_y \psi_j(x,y) = 0,\]
where $(c_j,\eta_j, \psi_j)$ corresponds to $(\lambda_j,w_j)$ via \eqref{E:c-eta} and \eqref{E:velocity}.

In summary, there is a sequence of solution triples
$\{(c_j, \eta_j, \psi_j )\} \subset \mathcal{C}$ 
in the space $\mathbb{R}_+ \times C^{3+\alpha}(\mathbb{R}) \times C^{3+\alpha}(\Omega_\eta)$ 
such that 
\[
\text{either}\quad\lim_{j \to \infty} c_j=\infty \quad
\text{or}\quad\lim_{j \to \infty}\sup_{\overline{\Omega}_{\eta_j}}\partial_y\psi_j = 0.
\]
This completes the proof.
\end{proof}

The remainder of this subsection is to refine the location of stagnation points 
in a ``limiting" solution of \eqref{E:main}.

\begin{proof}[Proof of Theorem \ref{T:main-}]
If $\gamma(r)\leq 0$ and $\gamma'(r)\geq 0$ for $0\leq r<\infty$, then 
by virtue of Lemma \ref{L:min-max}, the supremum of $\psi_y$ in $\overline{\Omega}_\eta$
is attained at the wave crest. Therefore, the only possible stagnation point is the wave crest.
This completes the proof. 
\end{proof}

In the finite-depth case \cite{CoSt07, Var09}, 
if the vorticity is non-positive (not necessarily monotone with depth), 
the only possible point of stagnation is shown to be the wave crest. 
In the infinite-depth case, the same result is expected. 
In other words, the monotonicity assumption in Theorem \ref{T:main-} is expected to be removed.
No proof is given presently, but some partial results are collected below.
 
Since $\Gamma(-\psi) \geq 0$ for $\gamma (r) \leq 0$, 
it follows by \eqref{E:velocity-bound1} that 
\[ |\psi_y(0, \eta(0)) | \leq |\nabla \psi(x,y)| \qquad \text{for all } (x,y) \in \Omega_{\eta}. \]
Hence, if one can show that $\psi_x=0$ at the point of $\psi_y=0$ then 
one can obtain the desired result. 
Unfortunately, $\psi_y$ or $w_p$ does not have a maximum principle, and hence 
it is not clear how to control the behavior of $w_p$ in terms of $w_q$.  

Alternatively, if the vorticity is non-positive, Lemma \ref{L:monotone} states 
that $\psi_y$ is monotone on the free surface. Thus, 
if one can show that the maximum of $\psi_y$ in $\overline{\Omega}_\eta$ occurs at the free surface 
then one can obtain the desired result. 
In the finite-depth case, it is shown \cite{CoSt07, Var09} that if the vorticity is non-positive 
then $\psi_y$ is monotone along the free surface, below the wave crest and the wave trough, 
and along the bottom. However, there is no sufficient control of the velocity at the infinite bottom. 

%we recall from Lemma \ref{L:loc-bifurcation-} that 
%the bifurcation condition \eqref{C:bifurcation} holds true. 
%Hence, a connected set $\mathcal{C}$ in $\mathbb{R}_+\times C^{3+\alpha}_{per}(\mathbb{R})
%\times C^{3+\alpha}_{per}(\overline{\Omega}_\eta)$ of solution triples $(c,\eta,\psi)$ exists 
%with the properties (i) and (ii) of Theorem \ref{T:main}.

%Let us assume that the first alternative 
%$\lim_{j \to \infty} \partial_y\psi_j (\pm L, \eta_j(\pm L))=-\infty$ in (ii) of Theorem \ref{T:main} occurs.
%By Lemma \ref{L:green}, it follows that 
%\[ |\partial_y\psi_j(\pm L, \eta_j(\pm L))|^{3/2} \leq c_jL + |\partial_y \psi_j (0,\eta_j(0))|^{3/2}\]
%for all $j$. On the other hand, by Lemma \ref{L:speed}, it follows that 
%\[ (\partial_y\psi_j)^2(0,\eta_j(0)) <c_j^2+2\Gamma_\infty \]
%for all $j$. Thus, 
%\[ |\partial_y\psi_j(\pm L, \eta_j(\pm L))|^{3/2} \leq c_jL + |c_j^2+2\Gamma_\infty|^{3/4}\]
%for all $j$, and the first alternative implies that $\lim_{j \to \infty} c_j =\infty$.

%Next, we assume that the second alternative 
%$\lim_{j \to \infty}\sup_{\overline{\Omega}_{\eta_j}}\partial_y\psi_j = 0$ occurs. 
%Since $\gamma(s)\leq 0$ it is immediate that $\Gamma(-\psi) \geq 0$. 
%By Lemma \ref{L:velocity-bound} then 
%\[ |\partial_y \psi_j(0, \eta_j(0)) | \leq |\nabla \psi_j(x,y)| \qquad \text{for all } (x,y) \in \Omega_{\eta_j}. \]
%Thus, this alternative implies that 
%$\lim_{j \to \infty} \partial_y \psi_j(0,\eta_j(0)) =0$. 

\begin{proof}[Proof of Theorem \ref{T:main+}]
Assume the second alternative 
$\lim_{j \to \infty}\sup_{\overline{\Omega}_{\eta_j}}\partial_y\psi_j = 0$ occurs.
We may choose a sequence $\{s_j\}$ and a sequence $\{(x_j, y_j)\}$ such that 
$s_j \to 0-$ as $j \to \infty$, $(x_j, y_j) \in \overline{\Omega}_{\eta_j}$ for each $j$ and 
$\partial_y \psi_j(x_j,y_j)=s_j$ for each $j$. 
%We may assume that $\|\partial_y \psi_j\|_{C^{2+\alpha}(\overline{\Omega}_{\eta_j})}$ is bounded;
%otherwise, we may use the discussion for the first alternative in (ii). 
%Then, $\partial_y^2\psi_j (x,y) \to 0$ as $y \to -\infty$ uniformly for $j$.
%Indeed, the rate of the convergence is at least the decay rate of $\gamma(s) \to 0$ as $s \to \infty$. 
We may assume that $\{y_j\}$ is bounded from below; otherwise, 
the stagnation occurs at the infinite bottom, and $\lim_{j \to \infty} c_j =0$. 

Differentiating the Poisson equation  $\Delta \psi_j=-\gamma(\psi_j)$ in the $y$-variable leads that 
\[\Delta \partial_y \psi_j + \gamma'(\psi_j)\partial_y \psi_j=0 
\qquad \text{in }\, \Omega_{\eta_j}\]
for each $j$. Let us introduce a sequence of functions
\begin{equation}\label{D:sn}
W_j(x,y)=\partial_y \psi_j(x,y) + s_j e^{\beta(y-y_j)} \qquad \text{for }(x,y)\in \overline{\Omega}_{\eta_j},
\end{equation}
where $\beta>0$ is a constant such that $\beta^2+\gamma'(r) \geq 0$ for all $0\leq r<\infty$.
It is straightforward that 
\[
\Delta W_j + \gamma'(\psi_j)W_j =s_j(\beta^2+\gamma'(\psi_j))e^{\beta(y-y_j)}\geq 0
\qquad \text{in }\Omega_{\eta_j}
\]
and that $W_j(x_j,y_j)=0$ for each $j$.

Since $\gamma'(\psi) \leq 0$, the weak maximum principle ensures that
$W_j$ in $\overline{\Omega}_{\eta_j}$ attains its maximum 
either on the surface or at the infinite bottom.
On the other hand, 
\[
W_j \to -c_j <0 \qquad \text{as }y \to -\infty \quad \text{for each $j$}.
\]
Therefore, $W_j$ in $\overline{\Omega}_{\eta_j}$ attains its maximum on the free surface $y=\eta_j(x)$.
Let $(\xi_j,\eta_j(\xi_j))$, $ -L \leq \xi_j \leq 0$ be the maximum point 
of $W_j$ in $\overline{\Omega}_{\eta_j}$. Since $W_j(x_j,y_j)=0$, it follows that  
\[ \partial_y \psi_j(\xi_j, \eta_j(\xi_j)) +s_j e^{\beta(\eta_j(\xi_j)-y_j)}\geq 0,
\]
whence
\[
0 \leq -\partial_y \psi_j(\xi_j, \eta_j(\xi_j))\leq s_j e^{\beta(\eta_j(\xi_j)-y_j)}.
\]
Since $\{y_j\}$ is bounded from below, by taking the limit as $j \to \infty$ we conclude that
\[
\lim_{j \to \infty} \partial_y \psi_j(\xi_j,\eta_j(\xi_j))=0.
\]
That means $\psi_y$ somewhere on the free surface becomes arbitrarily small. 
This proves (ii)~of Theorem~\ref{T:main+}.

In case of a non-negative vorticity, 
$\psi_y(x,\eta(x))$ is {\em not} necessarily monotone on the free surface,
and one cannot expect that $\psi_y=0$ occurs at the wave crest. 
Nevertheless, 
$\psi_y=0$ cannot occur at the wave trough unless the free surface is flat.
Indeed, if $(x_m, \eta(x_m))$ is the point of maximum horizontal velocity $\psi_y$ on $y=\eta(x)$ then
by Bernoulli's equation, it follows that 
\[ \psi_y^2(0, \eta(0)) +2g\eta(0) =\psi_x^2(x_m, \eta(x_m)) +\psi_y^2(x_m, \eta(x_m)) +2g\eta(x_m)),\]
whence
\[ \psi_x^2(x_m, \eta(x_m)) \geq 2g(\eta(0)-\eta(x_m)).\]

We assume, in addition, that $\psi_y(\pm L, \eta(\pm L))$ is bounded along $\mathcal{C}$. 
Then, by \eqref{trough}, the speed of wave propagation $c$ is bounded along $\mathcal{C}$,
and the first alternative in (ii) does not occur. 
If, in addition, $\gamma$ is sufficiently small so that \eqref{C:small} holds, i.e., 
\[g+\gamma(0) \inf_{\mathcal{C}} \psi_y(\pm L, \eta(\pm L)) \geq 0,\] then by Lemma \ref{L:monotone} 
the relative flow speed $\psi_y$ is monotone from crest to trough, and hence 
the second alternative in (ii) can be refined as 
\[\lim_{j \to \infty} \partial_y\psi_j(0, \eta(0))=0.\]
This completes the proof.
\end{proof}

The smallness condition \eqref{C:small} improves that \eqref{E:old-smallness} is non-positive 
\cite{Hur06}, which involves solutions through the function $n(x)$ in \eqref{E:n(x)}.

%We end our discussion by obtaining another bound, analogous to that in \cite{Tol96} 
%when the second alternative in (ii) occurs in case of a nonpositive monotone vorticity.

%\begin{proposition}
%Suppose $\gamma(s) \leq 0$ for all $s \in [0,\infty)$ and $\gamma$ is as in Theorem \ref{T:main}. 
%If the speed of wave propagation $c$ is bounded along the solution continuum $\mathcal{C}$ 
%of Theorem \ref{T:main}, and if 
%\[\max_{s \in [0,\infty)} |\gamma(s)| \leq \frac{2g}{\sup_{\mathcal{C}} c}\] 
%then any solution triple $(c,\eta, \psi) \in \mathcal{C}$ satisfies 
%\[ |\psi_y(0,\eta(0))|+|\psi_y(\pm L, \eta(\pm L))| \leq 2c. \]
%\end{proposition} 

%%%%%%%%%%%%%%%%%%%%%%%%%%%%%%%%%%%%%%%%%%%%%%%%%%%%%%%%%%%%%%%%%%%%%%%%%%%%%%%%%%%%%%%%%%%%%%%%%%%%%%%%%%%%%%%%%%%%%%%%%%%%%%%%%%%%%%%%%%%%%%%%%%%%%%%%%
\section{Reformulation via a quasi-conformal mapping}\label{S:Zei73}
%%%%%%%%%%%%%%%%%%%%%%%%%%%%%%%%%%%%%%%%%%%%%%%%%%%%%%%%%%%%%%%%%%%%%%%%%%%%%%%%%%%%%%%%%%%%%%%%%%%%%%%%%%%%%%%%%%%%%%%%%%%%%%%%%%%%%%%%%%%%%%%%%%%%%%%%%
In the irrotational setting, Stokes \cite{Sto80} proposed to use  
the (relative) velocity potential $\phi(x,y)$, 
defined in $\overline{\Omega}_\eta$ as $\phi_x=u-c$, $\phi_y=v$,
and the relative stream function $\psi(x,y)$ in order to study periodic traveling waves.
By the conformal\footnote{
Since $-\phi$ is the harmonic conjugate of $\psi$, 
the complext function $\phi+i\psi$ is holomorphic in $\Omega_\eta$.} 
hodograph transform $(x,y) \mapsto (\phi,\psi)$, 
the system (\ref{E:stream}b)-(\ref{E:stream}e) is recast as Nekrasov's integral equation \cite{Nek51}
(or Babenko's pseudo-differential equation \cite{Bab87, BuTo03}), 
and an a priori bound for speed of wave propagation follows.
%The existence theory of irrotational Stokes waves heavily rely on 
%special techniques in complex analysis or harmonic analysis.

With nontrivial vorticities, unfortunately, the velocity potential is not available. 
Nevertheless, under the no-stagnation assumption \eqref{E:no-stag}, 
the {\em pseudo-velocity potential} and a quasi-conformal transform offer 
an alternative reformulation of (\ref{E:stream}b)-(\ref{E:stream}e),
which share in common with the irrotational setting (\cite{LC25}, for instance) some structural similarity. 
The development is adapted from \cite{Zei73}. 

In preparation, let us rename the trivial solution \eqref{E:trivial} as 
\[  c-u_{tr}(\psi)=\exp(\tau_{tr}(\psi)) = (\lambda+2\Gamma(-\psi))^{1/2},\]
and let us make the ansatz 
\[ u(x,y)-c=-\exp(\tau_{tr}+\tau) \cos \theta, \qquad v(x,y)=\exp(\tau_{tr}+\tau) \sin \theta,\]
or equivalently, 
\[ e^{2\tau} =\frac{(u-c)^2+v^2}{(u_{tr}-c)^2}, \qquad \tan \theta=-\frac{v}{u-c}.\]
By construction, $\exp(2\tau_{tr}+2\tau)=(u-c)^2+v^2$ measures the kinetic energy density of the flow
and $\theta$ on the free surface $y=\eta(x)$ measures the angle 
that the wave profile makes with the positive horizontal direction.
Note that $\exp(2\tau_{tr}+2\tau) >0$ for a regular wave 
and a stagnation point corresponds to where $\exp(2\tau_{tr}+2\tau)=0$. 
In consideration of Stokes waves, $\tau$ is required to be even and $2L$-periodic in the $x$-variable, 
and $\theta$ is required to be odd and $2L$-periodic in the $x$-variable.

Let us define the pseudo-velopcity potential $\phi(x,y)$ in $\Omega_\eta$ by 
\begin{equation}\label{D:phi}
\phi_x=W(x,y)\psi_y=W(u-c), \qquad \phi_y=-W(x,y)\psi_x=Wv
\end{equation}
and $\phi(0,\eta(0))=0$ for some function $W(x,y)$.
It is straightforward that (\ref{E:stream}b) dictates that the auxiliary function $W$ satisfies
\begin{equation}\label{E:D}
W_x\psi_x+W_y\psi_y=W\gamma(\psi) \qquad \text{in}\quad 0<y<\eta(x)
\end{equation}
It is reasonable to require that $W(x,y) \to 1$ as $y \to -\infty$. 
In addition, $W$ is required to be positive, even and $2L$-periodic in the $x$-variable.
The complex potential $\phi+i\psi$ is a $p$-analytic function. 
In the irrotational setting, $W(x,y)=1$ everywhere in the fluid region
and $\phi+i\psi$ is a holomorphic function.  
Note that $\phi$ is odd and $\psi$ is even in the $x$-variable. Moreover, 
\[\phi(L,y)-\phi(-L,y)=\int^{L}_{-L} (u(x,y)-c)dx = -cL,\]
independently of $y$.  By the oddness of $\phi$,  
\[ \phi(\pm L, y)=\mp cL, \qquad \phi(0,y)=0 \qquad \text{for all $0<y<\eta(x)$}\] 
and $\phi(x+2L,y)=\phi(x,y)-c L$.

Let us define the independent variables 
\begin{equation}\label{D:independent}
q^*=-\phi(x,y) \quad \text{and}\quad p=-\psi(x,y).
\end{equation}
They map the fluid region of one period $\{ (x,y) \in \Omega_\eta: -L<x<L \}$ into the semi-infinite strip 
$(-c L, c L) \times ( -\infty, 0)$ in the $(q^*,p)$-plane and the free surface
$\{ (x, \eta(x)): -L<x<L \}$ to the horizontal line segment $(-cL, cL) \times \{ 0\}$.
In what follows, let $$R^*=\{(q^*,p): -c L<q^*< c L\, , \, -\infty<p< 0\}$$
be the domain of one period in the transformed variables.

The no-stagnation assumption, $\exp(2\tau_{tr}+2\tau)>0$ throughout the fluid region guarantees 
the mapping $(x,y) \mapsto (\phi, \psi)$ is quasi-conformal \cite{Zei73}. 
In the irrotational setting, the mapping is conformal.
Furthermore, under this physically motivated stipulation, 
$W$ is uniquely solvable in the $C^1$ class provided that $\psi$ is in the $C^2$ class.

In view of $\tau$ and $\theta$ as functions of $q^*$ and $p$, 
under the change of variables \eqref{D:independent}, 
straightforward calculations yield that (\ref{E:stream}b) 
translates into the following inhomogeneous Cauchy-Riemann equations
in the rectangle $R$ as
\begin{equation}\label{E:CR3}
\begin{split}
- \hspace{-.15in}&W\tau_{q^*}+\theta_p=0, \\
&W\theta_{q^*}+\tau_p=-\frac{d\tau_{tr}}{dp}+ \exp(-2(\tau_{tr}+\tau))\gamma(-p).
\end{split}\end{equation}
Furthermore, \eqref{E:D} becomes 
\[ e^{2(\tau_{tr}+\tau)} W_\psi=W\gamma(\psi), \qquad  W(q^*,p) \to 1 \quad \text{as $p \to -\infty$}.\]
Its unique solution is given explicitly in the $(q^*,p)$-plane as 
\begin{equation}
W(q^*,p)=\exp\left(-\int_{-\infty}^p \exp(-2(\tau_{tr}+\tau)) \gamma(-p)dp\right),
\end{equation}
provided that $\exp(2\tau_{tr}+2\tau)>0$.
The coefficient function $W$ of \eqref{E:CR3} is a nonlocal operator 
involving the dependent variable $\tau$.
In particular, \eqref{E:CR3} does not enjoy the maximum principle, 
and it may not be suitable for global existence theory. 

The change of variables in \eqref{D:stream} and \eqref{D:phi} is written in the concise form as
$$\left( \begin{matrix} \phi_x & \phi_y \\ \psi_x & \psi_y \end{matrix}\right)
=\left( \begin{matrix} W(u-c) & Wv \\ -v& u-c\end{matrix}\right).$$
The back-transformation is given by 
$$\left( \begin{matrix} x_\phi & x_\psi \\ y_\phi & y_\psi \end{matrix}\right)
=\frac{1}{W((u-c)^2+v^2)} \left( \begin{matrix} u-c & -Wv \\ v& W(u-c)\end{matrix}\right).$$

A main advantage of this approach is that
the nonlinear boundary condition on the free-surface takes a convenient form.
Indeed, differentiation of the Bernoulli equation (\ref{E:stream}d) 
with respect to $q^*$-variable yields the boundary condition 
\[
\tau_{q^*}=g\lambda^{-3/2} W(0, \tau)^{-1}e^{-3\tau} \sin \theta \qquad \text{for}\quad p=0. \]
It uses that $e^{\tau_{tr}(0)} =\lambda^{1/2}$. Since $-W\tau_{q^*}+\theta_p=0$, furthermore,
\begin{equation}\label{E:bernoulli}
\theta_p=g\lambda^{-3/2}e^{-3\tau}\sin \theta \qquad \text{for}\quad p=0.
\end{equation}

The boundary condition \eqref{E:bernoulli} is the same as that in the irrotational setting \cite{LC25}.
In the irrotational setting, $\theta$ satisfies the Laplace equation $\Delta \theta =0$ in $R^*$ 
and the boundary condition \eqref{E:bernoulli}. 
By the sine series methods, then it leads to {\em Nekrasov's integral equation}
\[ \theta(s) =\frac{1}{3\pi} \int^\pi_0 \log\left| \frac{\sin \frac12(s+t)}{\sin \frac12(s-t)} \right|
\frac{\sin \theta(t)}{ \nu^{-1}+\int^t_0 \sin \theta(t')dt'} dt,\]
where $\nu=3gLc/\pi\lambda^{3/2}$. In view of the positivity of the kernel, 
by multiplying the equation by $\sin s$ and integrating over $(0,\pi)$ then yields that 
\begin{align*}
\int^\pi_0 \theta(s)\sin s ds=&\frac{1}{3\pi}\int^\pi_0
 \frac{\sin \theta(t)\sin t}{ \nu^{-1}+\int^t_0 \sin \theta(t')dt'} dt \\
 <&\frac{\nu}{3\pi} \int^\pi_0\theta(t) \sin t dt,
\end{align*}
which implies that $\mu>1/3\pi$. Therefore,  
an a priori bound for $\lambda$ (or $\nu$) is obtained (\cite[Chapter 10]{BuTo03}, for instance).

%The system \eqref{E:CR3}-\eqref{E:bernoulli3} is equivalent to the original system 
%\eqref{E:stream},
%the proof of which is found in \cite[Section 2.4]{Zei73}. 
%The system, in turn, is equivalent to \eqref{h-prob}. Indeed, $\tau$ and $\theta$
%are related to $h$ via
%\begin{equation}\label{E:theta}
%\tau=\frac{1}{2} \log \left(\frac{1+h_q^2}{h_p^2}\right), \qquad 
%\theta=-\arctan h_q.\end{equation}

\subsection*{Acknowledgment}
This work is partly supported by the NSF grants DMS-0707647 and DMS-100254.

\bibliographystyle{amsalpha}
\bibliography{steadyWW}

\end{document}